\documentclass[12pt]{article}

\usepackage{latexsym, amsfonts, amscd, amssymb, verbatim, amsmath,enumerate}  
\usepackage{mathtools}
\usepackage{comment}
\usepackage{enumitem}

\usepackage[sorting=anyt]{biblatex}

\usepackage[capitalise,nameinlink]{cleveref}
\numberwithin{equation}{section}
\newlist{assumption}{enumerate}{1}
\setlist[assumption]{label=(\textsc{a}\arabic*)}
\crefname{assumptioni}{Assumption}{Assumptions}

\newtheorem{theorem}{Theorem}[section]
\newtheorem{proposition}{Proposition}[section]

\newtheorem{lemma}{Lemma}[section]
\newtheorem{remark}{Remark}[section]

\newcommand{\R}{\mathbb{R}}

\newcommand{\cone}{\text{cone}}
\newcommand{\norm}[1]{||#1||}
\newcommand{\proj}{\text{proj}}
\newcommand{\Lip}{\text{Lip}}

\newcommand{\qed}{\hfill\square}

\newcommand{\proofbox}{\hspace{\fill}{$\Box$}}
\newenvironment{proof}{\textbf{Proof}.}{\proofbox}

\usepackage{xcolor}
\usepackage[textsize=footnotesize]{todonotes}

\renewenvironment{abstract}
{\par\noindent\textbf{\abstractname.}\ \ignorespaces}
{\par\medskip}

\date{}

\def\bar{\overline}

\date{}

\begin{document}

\title{Lipschitz continuity of solutions and corresponding multipliers to distributed and boundary semilinear elliptic optimal control problems with mixed pointwise control-state constraints}

\author{V. H. Nhu\footnote{
		Faculty of Fundamental Sciences, PHENIKAA University, Yen Nghia, Ha Dong, Hanoi 12116, Vietnam; email: nhu.vuhuu@phenikaa-uni.edu.vn};
	N. Q. Tuan\footnote{
		Department of Optimization and Control Theory, Institute of Mathematics, Vietnam Academy of Science and Technology, 18 Hoang Quoc Viet road, Hanoi, Vietnam and Department of Mathematics, Hanoi Pedagogical University 2, Xuan Hoa, Phuc Yen, Vinh Phuc, Vietnam; email: nguyenquoctuan@hpu2edu.vn}; 
	N. B. Giang\footnote{
		Department of Information and Technology, National University of Civil Engineering, 55 Giai
		Phong Str., Hanoi, Vietnam; email: giangnb@nuce.edu.vn}; and
	N. T. T. Huong\footnote{
		School of Applied Mathematics and
		Informatics, Hanoi University of Science and Technology, 1 Dai Co
		Viet, Hanoi, Vietnam; email: huong.nguyenthithu3@hust.edu.vn}
}

\maketitle

\medskip


\begin{abstract}
	This paper is concerned with the existence and regularity of mininizers as well as of corresponding multipliers to an optimal control problem governed by semilinear elliptic equations, in which mixed pointwise control-state constraints are considered in a quite general form and the controls act simultaneously in the domain and on the boundary. Under standing assumptions, the minimizers and the corresponding multipliers do exist. Furthermore, by applying the bootstrapping technique and establishing some calculation tools for functions in Sobolev spaces of fractional order, the optimal solutions and the associated Lagrange multipliers are shown to be Lipschitz continuous.
\end{abstract}




\section{Introduction}
Let $\Omega$ be a bounded domain in $\mathbb{R}^N$ with the boundary
$\Gamma$ of class $C^{1,1}$ and $N\geq 2$. We consider the following
semilinear elliptic optimal control problem with mixed pointwise
constraints:
Find a couple of control functions $\left(u,v\right) \in L^{p}\left(\Omega\right) \times L^{q}\left(\Gamma\right)$ with $p > N/2$, $q > N-1$, $p, q \geq 2$
and a corresponding state function $y\in H^1\left(\Omega\right) \cap C(\overline\Omega)$, which minimize the cost functional 
\begin{subequations}
	\makeatletter
	\def\@currentlabel{P}
	\makeatother
	\label{eq:P}
	\renewcommand{\theequation}{P.\arabic{equation}}   
	\begin{align}
		 I\left(y, u,v\right)&=\int_\Omega \left(L\left(x, y\left(x\right)\right)+\frac{\lambda_1}2|u\left(x\right)|^2+\frac{\lambda_2}{p}|u\left(x\right)|^{p}\right)dx \notag \\
		\MoveEqLeft[-6] + \int_\Gamma \left(\ell \left(x,y\left(x\right)\right) +\frac{\mu_1}2|v\left(x\right)|^2 +\frac{\mu_2}{q}|v\left(x\right)|^{q}\right)d\sigma(x), \label{eq:objective-func} \\
		\intertext{subject to}
		&\begin{cases}
			Ay + f\left(x,y\right) = u \quad &{\rm in}\ \Omega\\
			\partial_{\nu_A} y = v \quad &{\rm on}\ \Gamma
		\end{cases} \label{eq:state}\\
		\intertext{and}
		& g_1\left(x,y\left(x\right)\right) + \zeta_1\left(u\left(x\right)\right) \leq 0\ {\rm  a.e.}\ x\in\Omega,
		\label{eq:constraint-domain}\\
		& g_2\left(x',y\left(x'\right)\right) +\zeta_2\left(v\left(x'\right)\right) \leq 0\ {\rm  a.e.}\ x'\in\Gamma,
		\label{eq:constraint-boundary}
	\end{align}
\end{subequations}
where $\lambda_i$ and $\mu_i$, $i=1,2$, are positive numbers; $A$ is a uniformly elliptic operator in $\Omega$ and $\partial_{\nu_A}$ stands for the conormal-derivative associated with $A$;
$L, f : \Omega \times \mathbb R \to \mathbb R$ and $\ell :\Gamma \times \mathbb R \to \mathbb R$ are Carath\'{e}odory
functions; $g_1: \Omega \times \mathbb{R} \to \mathbb{R}$ and $g_2: \Gamma \times \mathbb{R} \to \mathbb{R}$ are continuous; $\zeta_i: \mathbb{R} \to \mathbb{R}$, $i=1,2$, are continuously differentiable and strictly monotonic. 
The specific assumptions imposed on these functions shall be presented in \cref{sec:assumption}.

It is well-known that Lagrange multipliers in the optimality conditions for optimal control problems with pointwise state constraints are measures rather than functions; see, e.g. \cite{Casas1986,Casas1993,CasasRaymondZidani1996}. However, for a practical purpose such as in numerical analysis, these multipliers need to exhibit a better regularity property, for instance, $L^d$-regularity with some $d \in [1,\infty]$ and/or $C^{0,1}$-regularity. 
In order to derive this higher regularity, the regularization approach is naturally applied. In our problem, the constraints \eqref{eq:constraint-domain} and \eqref{eq:constraint-boundary}, respectively, aim at regularizing the pointwise pure state constraints in the domain and on the boundary.
Meanwhile, the objective functional is a regularization of the corresponding quadratic form of controls.

\medskip


Let us comment on related works. We first investigate some literature on optimal control problems where the controls act only on the domain. A class of quadratic optimization problems in Hilbert spaces was investigated in \cite{Troltzsch2005}, in which pointwise box-constraints and constraints of bottleneck type were considered. The Lagrange multipliers associated with these problems were proved to exist and belong to $L^2$-spaces. An improved result on the regularity of multipliers was shown in \cite{RoschTroltzsch2006} for optimal control problems governed by semilinear elliptic equations with mixed pointwise control-state constraints of bottleneck type on the domain. 
Using the idea of linearization and application of duality theory, the authors in \cite{RoschTroltzsch2006} pointed out that the Lagrange multipliers are in fact bounded on the whole domain. Under natural assumptions, the Lipschitz continuity regularity of optimal controls and the corresponding multipliers was confirmed by R\"{o}sch and Tr\"{o}ltzsch \cite{RoschTroltzsch2007} when considering semilinear optimal control problems with mixed control-state constraints in a quite general form  for controls belonging to $L^\infty(\Omega)$-space. The key techniques exploited in \cite{RoschTroltzsch2007} are the Yosida-Hewitt theorem and bootstrapping arguments.

Regarding the optimal control problems with controls acting both in the domain and on the boundary, the research topic on the regularity of multipliers has been recently studied in \cite{GiangTuanSon2020,KienHuong2021}. 
The authors in \cite{GiangTuanSon2020} considered the mixed control-state constraints that have a general form like \cite{RoschTroltzsch2007}. There, by exploiting the representation of functional on $L^\infty$-spaces by a density in $L^1$-spaces (see \cite[Prop.~5, Chap.~8]{Ioffe}) and the bootstrapping technique, they showed the Lipschitz continuity of  Lagrange multipliers under the assumption that the mentioned optimal state is assumed to be Lipschitz continuous on the whole domain. Kien et al. \cite{KienHuong2021} investigated a special case of \eqref{eq:P} for $N\in \{2,3\}$, $p=q=4$ and $\zeta_1(t) = \epsilon_1 t + t^3, \zeta_2(t) = \epsilon_2 t + t^3$ with constants $\epsilon_1, \epsilon_2 >0$. Using the  metric projection and the bootstrapping technique, Kien et al. showed that the minimizers and the corresponding Lagrange multipliers are H\"{o}lder continuous.

\medskip
The goal of this paper is to study the Lipschitz continuity regularity of optimal solutions and the corresponding Lagrange multipliers of problem \eqref{eq:P}. To this aim, we first verify the Robinson constraint qualification \cite{Robinson1976} and thus derive the existence of Lagrange multipliers according to the classical theory of mathematical programming problems \cite{BonnansShapiro2000}. After that, as in \cite{KienHuong2021}, we employ the  metric projection and the bootstrapping technique to show the H\"{o}lder regularity. Finally, to derive a higher regularity property--the Lipschitz continuity--we need to establish some calculation tools on the Sobolev spaces of fractional order such as the chain rule (see \cref{lem:chain-rule,rem:chain-rule} below) and the product of functions (see \cref{lem:production} below). 

Our main results greatly improve the ones in \cite{KienHuong2021} and in \cite{GiangTuanSon2020}. Although the constraints are quite general and the set of assumptions is natural, we still can relax the Lipschitz continuity assumption on the considered optimal state as compared to \cite{GiangTuanSon2020}.

Finally and interestingly, let us emphasize that  $\zeta_1$ and $\zeta_2$ are in general not  affine functions with respect to $u$ and $v$, respectively. The associated Nemytskii operators are thus not differentiable as functions from $L^d$-space to itself for any $d \in [1, \infty)$ (see \cite{Troltzsch2010}). This is  in contrast to the differentiability of Nemytskii operators corresponding to mixed pointwise
constraints of optimal control problems with controls being investigated in $L^\infty$-spaces; see, e.g.   \cite{RoschTroltzsch2006,RoschTroltzsch2007,GiangTuanSon2020,Griesse2010}. Our work then extends the frequently used setting of $L^\infty$-controls to   $L^p$-controls. 
\medskip

The paper is organized as follows. This introduction ends with some notation. 
\cref{sec:control2state-OC} then gives a precise statement of all standing assumptions used in the whole paper and presents the existence of a global optimal solution to \eqref{eq:P} and a  system of corresponding optimality conditions. Finally, \cref{sec:Lipschitz-reg} is devoted to the main result--the Lipschitz continuity of optimal solutions and Lagrange multipliers.

\paragraph*{Notation.}
Given Banach spaces $X$ and $Y$, the notation $X \hookrightarrow Y$  means that $X$ is continuously embedded in $Y$, and $X \Subset Y$ indicates that $X$ is compactly embedded in $Y$. 
For a given point $u\in X$ and $\rho>0$, $B_X(u,\rho)$ and $\overline B_X(u,\rho)$ stand for the open and closed balls of radius $\rho$ centered at $u$, respectively. For a Banach space $X$, $X^*$ stands for the dual space of $X$ and $\langle \cdot, \cdot \rangle_{X^*,X}$ denotes the duality product between $X^*$ and $X$. For $t\in [1,\infty]$, by $t'$ we denote the conjugate of $t$, that is,
\begin{equation*}
	t' = \begin{cases}
	\infty & \text{if } t=1,\\
	1 & \text{if } t=\infty,\\
	\frac{t}{t-1} & \text{otherwise.}
	\end{cases}
\end{equation*}
Given a Carath\'{e}odory function $g$, with a little abuse of notation the associated Nemytskii operator defined on $L^d(\Omega)$ for some $d \in [1,\infty]$ is also named by the same symbol.
Finally, $C$ stands for a generic positive constant, which might be different at different places of occurrence. We also write, e.g., $C(\xi)$ or $C_\xi$ for a constant dependent only on the parameter $\xi$.

\section{Standing assumptions, existence of optimal solutions, and first-order optimality conditions} \label{sec:control2state-OC}

\subsection{Standing assumptions} \label{sec:assumption}
Throughout the article, we need the following standing assumptions for problem \eqref{eq:P}.

\begin{assumption}
	\item \label{ass:domain} $\Omega \subset \R^N$, $N \geq 2$ is a bounded domain with a $C^{1,1}$-boundary $\Gamma$ and there hold that $p > N/2$, $q > N-1$, $p \geq 2$, $q\geq 2$, $\lambda_i, \mu_i >0$ for all $i = 1,2$. 
	
	\item \label{ass:A-operator} 
	The operator $A$ is uniformly elliptic in $\Omega$ and is defined by 
	\[
		Ay\left(x\right)=-\sum_{i,j=1}^N D_j\left(a_{ij}\left(x\right)D_iy\left(x\right)\right)+a_0\left(x\right)y\left(x\right),
	\]
	where coefficients $a_{ij} \in C^{0,1}\left(\bar \Omega\right)$ satisfying $a_{ij}\left(x\right)=a_{ji}\left(x\right)$ for all $1 \leq i,j\leq N$, $a_0 \in L^\infty(\Omega)$, $a_0\left(x\right)\geq 0$ on $\Omega$, and $a_0(x) >0$ on a set of positive measure. Here we assume that there exists $\alpha >0$ such that
	\[
		\alpha |\xi |^2 \leq \sum_{i,j=1}^N a_{ij}\xi_i\xi_j \quad \text{for all } \xi=\left(\xi_1,\xi_2,...,\xi_N\right) \in \mathbb{R}^N.
	\]
	The symbol $\partial_{\nu_A}$ stands for the conormal-derivative associated with $A$ and is defined via
	\[ 
		\partial_{\nu_A} y\left(x\right) = \sum_{i,j=1}^N a_{ij}\left(x\right)D_iy\left(x\right)\nu_j\left(x\right)
	\] 
	with $\nu\left(x\right)=\left(\nu_1\left(x\right),...,\nu_N\left(x\right)\right)$ denoting the unit outward normal to $\Gamma$ at the point $x$. Hereafter, the measure $\sigma$ on the boundary $\Gamma$ is the usual $\left(N-1\right)$-dimensional measure induced by the parametrization (see, e.g., \cite{Kufner1977,Troltzsch2010,Evans1992}). 
	
	\item \label{ass:L-func} The functions $L:\Omega\times \mathbb R \to \mathbb R$ and $\ell:\Gamma\times \mathbb R \to \mathbb R$ are Carath\'{e}odory functions such that $L\left(x, \cdot\right)$ and $\ell\left(x', \cdot\right)$ are of class $C^1$ for each fixed $x\in\Omega$ and $x'\in\Gamma$. Furthermore, for each $M>0$, there exist positive numbers $K_{L, M}$ and $K_{\ell, M}$ such that
	\begin{align*}
		&|L\left(x,y_1\right)-L\left(x,y_2\right)| +|L_y' \left(x,y_1\right)-L_y'\left(x,y_2\right)| \leq K_{L,M}|y_1-y_2|,\\
		&|\ell\left(x',y_1\right)-\ell\left(x',y_2\right)| +|\ell_y' \left(x',y_1\right)-\ell_y'\left(x',y_2\right)| \leq K_{\ell, M}|y_1-y_2|
	\end{align*}
	for  a.e. $x \in \Omega$, $x'\in\Gamma$ and all $y_i\in\mathbb R$ satisfying $|y_i| \leq M$ with $ i=1,2$. Besides, $L\left(\cdot, 0\right)\in L^1\left(\Omega\right)$, $\ell\left(\cdot, 0\right)\in L^1\left(\Gamma\right)$, $L_y' \left(\cdot,0\right) \in L^\infty(\Omega)$, and $\ell_y' \left(\cdot,0\right) \in L^\infty(\Gamma)$. Moreover, there hold $L(x, y)\geq 0$ and $\ell(x', y)\geq 0$ for a.e. $x\in\Omega$, $x'\in\Gamma$ and all $y\in\mathbb{R}$.   
	
	\item \label{ass:f-func} The function $f:\Omega\times \mathbb R \to \mathbb R$ is a Carath\'{e}odory function such that $f\left(x, \cdot\right)$ is of class $C^1$ for each fixed $x\in\Omega$, and satisfies the property that for each $M>0$, there exists a positive number $K_{f, M}$ such that
	\begin{align*}
		&|f\left(x,y_1\right)-f\left(x,y_2\right)| +|f_y' \left(x,y_1\right)-f_y'\left(x,y_2\right)| \leq K_{f,M}|y_1-y_2|
	\end{align*}
	for  a.e. $x \in \Omega$ and all $y_i\in\mathbb R$ satisfying $|y_i| \leq M$ with $ i=1,2$. Besides, $f\left(\cdot, 0\right)=0$, $f_y'\left(x, y\right)\geq 0$ for  a.e. $x\in\Omega$ and all $y\in\mathbb{R}$. 
	
	\item \label{ass:g-funcs} The functions $g_1:\Omega\times \mathbb R \to \mathbb R$ and $g_2:\Gamma\times \mathbb R \to \mathbb R$ are continuous and for each fixed $x\in\Omega$ and $x'\in\Gamma$, $g_1\left(x, \cdot\right)$ and $g_2\left(x', \cdot\right)$ are of class $C^1$, and satisfy the property that for each $M>0$, there exist positive numbers $K_{g_1, M}$ and $K_{g_2, M}$ such that
	\begin{align*}
		&|g_1\left(x_1,y_1\right)-g_1\left(x_2,y_2\right)| +|g_{1y}'\left(x_1,y_1\right)-g_{1y}'\left(x_2,y_2\right)| \leq K_{g_1,M}\left(|x_1-x_2|+|y_1-y_2|\right),\\
		&|g_2\left(x'_1,y_1\right)-g_2\left(x'_2,y_2\right)| +|g_{2y}'\left(x'_1,y_1\right)-g_{2y}'\left(x'_2,y_2\right)| \leq K_{g_2,M}\left(|x'_1-x'_2|+|y_1-y_2|\right),
	\end{align*}
	for a.e. $x_i \in \Omega$, $x_i'\in\Gamma$ and all $y_i\in\mathbb R$ satisfying $|y_i| \leq M$ with $ i=1,2$. Furthermore, $g_1\left(\cdot, 0\right)= 0, g_2(\cdot, 0) =0$ and $g_{iy}'\left(\cdot, 0 \right)$, $i=1,2$, are $L^\infty$- functions. 
	\item \label{ass:zeta-funcs} The functions $\zeta_i:\mathbb R \to \mathbb R, i=1,2$, are strictly monotone and continuously differentiable such that $\zeta_i\left(0\right)=0$, and for each $M>0$ there exist positive numbers $K_{\zeta_i, M}$ satisfying
	\begin{align}\label{eq:lipzeta}
		&|\zeta_i\left(r_1\right)-\zeta_i\left(r_2\right)| +|\zeta'_{i}\left(r_1\right)-\zeta'_{i}\left(r_2\right)| \leq K_{\zeta_i,M}|r_1-r_2|
	\end{align}	for all $r_i \in \mathbb{R}$ with $|r_i| \leq M$ and $|\zeta_i'\left(\cdot\right)| \geq \rho_i>0$, $i=1,2$. Furthermore, there hold
	\begin{equation} \label{eq:Robinson-needed}
		\left\{
		\begin{aligned}
			& f_y'\left(x,\cdot\right)+\dfrac{g_{1y}'\left(x,\cdot\right)}{\zeta_1' \left( H_1\left(-g_1\left(x,\cdot\right)\right)\right)}\geq 0\\
			& \dfrac{g_{2y}'\left(x',\cdot\right)}{\zeta_2' \left( H_2\left(-g_2\left(x',\cdot\right)\right)\right)}\geq 0
		\end{aligned}
		\right.
	\end{equation} for a.e. $x\in\Omega$ and $x' \in \Gamma$, where $H_i$ is the inverse of $\zeta_i$, $i=1,2$.
\end{assumption}

\medskip

We finish this subsection with presenting the definition of the Sobolev spaces of fractional order on the boundary $\Gamma$; see. e.g. \cite{Adams,Kufner1977}. Let $W^{\tau,k}(\Gamma)$ with $\tau \in (0,1)$ and $k\geq 1$ be the subspace of all functions $v \in L^k(\Gamma)$ such that
\begin{equation} \label{eq:fractional-term}
	I_{\tau, k}(v) := \iint_{\Gamma \times\Gamma} \frac{|v(x)-v(x')|^k}{|x-x'|^{N-1+\tau k}}d\sigma(x)d\sigma(x') < \infty 
\end{equation}
equipped with the norm
\begin{equation*}
	\norm{v}_{W^{\tau,k}(\Gamma)} := \left\{ \norm{v}_{L^k(\Gamma)}^k +  I_{\tau, k}(v) \right\}^{\frac{1}{k}}. 
\end{equation*}
For $\tau \in (0,1)$ and $k >1$, we use the symbol $W^{-\tau,k}(\Gamma)$ for the dual space to $W^{\tau,k'}(\Gamma)$.

From now on, let $\gamma$ denote the trace operator from $W^{1,k}(\Omega)$ onto $W^{1-\frac{1}{k},k}(\Gamma)$ for $k>1$; see, e.g., \cite[Chap.~6]{Kufner1977}.

\subsection{The control-to-state operator} \label{sec:control2state-oper}

According to \cite[Thm.~3.1]{Casas1993}, there exists for each couple $(u,v) \in L^p(\Omega) \times L^q(\Gamma)$ a unique solution $y =y_{u,v} \in H^1(\Omega) \cap C(\overline\Omega)$  to \eqref{eq:state}. This defines a mapping, denoted by $S$, from $L^p(\Omega) \times L^q(\Gamma)$ to $H^1(\Omega) \cap C(\overline\Omega)$ by $S(u,v) = y_{u,v}$. Moreover, there is a constant $c_\infty$ independent of $u,v$, and $f$ such that
\begin{equation}
	\label{eq:apriori-esti-state}
	\norm{S(u,v)}_{H^1(\Omega)} + \norm{S(u,v)}_{C(\overline\Omega)} \leq c_\infty \left[ \norm{u}_{L^p(\Omega)} + \norm{v}_{L^q(\Gamma)} \right]
\end{equation}
(see, e.g., \cite[Thm. 4.7]{Troltzsch2010}). On the other hand, since $p > N/2$ and $q > N-1$, the embeddings $L^p(\Omega) \Subset (W^{1,r'}(\Omega))^*$ and $L^q(\Gamma) \Subset W^{-\frac{1}{r},r}(\Gamma)$ are compact for some $r > N$. From this and \cite[Lem.~2.4]{CasasMateos2002}, we derive the following implication
\begin{equation*}
	\left\{
	\begin{aligned}
		& u_k \rightharpoonup u \quad \text{in } L^p(\Omega)\\
		& v_k \rightharpoonup u \quad \text{in } L^q(\Gamma)
	\end{aligned}
	\right. 
	\, \implies \, S(u_k, v_k) \to S(u,v) \, \text{in } W^{1,r}(\Omega) \, \text{and thus in } H^1(\Omega) \cap C(\overline\Omega).
\end{equation*}
In conclusion, we have the following.
\begin{lemma}[cf. {\cite[Thm.~3.1]{Casas1993}}]
	\label{lem:control2state-compact}
	The control-to-state operator $S: L^p(\Omega) \times L^q(\Gamma) \to H^1(\Omega) \cap C(\overline\Omega)$ is completely continuous.
\end{lemma}

The following lemma provides the continuous differentiability of $S$ and the regularity of solutions to the linearization equation of the state equation. 
\begin{lemma}
	\label{lem:control2state-diff}
	The control-to-state operator $S: L^p(\Omega) \times L^q(\Gamma) \to H^1(\Omega) \cap C(\overline\Omega)$ is continuously differentiable in any $(u,v) \in L^p(\Omega) \times L^q(\Gamma)$. Furthermore, for any $(u,v), (\tilde{u},\tilde{v}) \in L^p(\Omega) \times L^q(\Gamma)$, $w := S'(u,v)(\tilde u,\tilde v)$ uniquely satisfies the equation
	\begin{equation}
		\label{eq:diff-S}
		\left\{
		\begin{aligned}
			Aw + f'_y(x, y_{u,v})w &= \tilde{u} \quad \text{in } \Omega,\\
			\partial_{\nu _A} w & = \tilde{v} \quad \text{on } \Gamma
		\end{aligned}
		\right.
	\end{equation}
	with $y_{u,v} := S(u,v)$.	
	Moreover, for any  $z=(u,v), \tilde{z} = (\tilde{u},\tilde{v}) \in L^p(\Omega) \times L^q(\Gamma)$, there holds
	\begin{equation}
		S'(z)\tilde{z} \in W^{1,r}(\Omega)
	\end{equation}
	with
	\begin{equation} \label{eq:exponent-r}
		r = \begin{cases}
			\min\left\{ \frac{pN}{N-p}, \frac{Nq}{N-1} \right \} &\quad \text{if } \frac{N}{2} < p < N,\\
			\frac{Nq}{N-1}& \quad \text{if } p \geq N.
		\end{cases}
	\end{equation}
	%
\end{lemma}
\begin{proof}
	The proof of the differentiability of $S$ is elementary and is similar to that in \cite[Chap.~4]{Troltzsch2010} (see also \cite[Thm.~3.3]{Casas1993}). We now show the regularity of solutions to \eqref{eq:diff-S}. To this end, we see from the Sobolev embeddings \cite{Adams,Kufner1977} that
	\begin{align*}
		&L^p(\Omega) \hookrightarrow \left( W^{1,r_1'}(\Omega)\right)^* \quad  \text{for } r_1 = \frac{Np}{N-p} \, \text{if } \frac{N}{2} < p < N; r_1 \geq 1  \, \text{if } p \geq N\\
		\intertext{and}
		& L^q(\Gamma) \hookrightarrow  W^{-\frac{1}{r_2}, r_2} (\Gamma)\quad \text{for } r_2 = \frac{Nq}{N-1}.
	\end{align*}
	Setting $r := \min\{r_1, r_2\}$ yields \eqref{eq:exponent-r}, $\tilde{u} \in  \left( W^{1,r'}(\Omega)\right)^*$ and $\tilde{v} \in W^{-\frac{1}{r}, r}(\Gamma)$. \cite[Lem.~2.4]{CasasMateos2002} then gives the desired conclusion.
\end{proof}

\subsection{Existence of optimal solution}
From now on, we denote by $\Phi$ the set of all couples $(u,v) \in L^p(\Omega) \times L^q(\Gamma)$ that together with $y:= S(u,v)$ satisfy \eqref{eq:constraint-domain} and \eqref{eq:constraint-boundary}, i.e.,
\begin{equation}
	\label{eq:admissible-set}
	\Phi := \left\{(u,v) \in L^p(\Omega) \times L^q(\Gamma) \mid (S(u,v),u,v) \, \text{satisfies  \eqref{eq:constraint-domain}--\eqref{eq:constraint-boundary}}\right\}.
\end{equation}

Although the set of admissible points of \eqref{eq:P} might not be bounded, the existence of global minimizers of \eqref{eq:P} is guaranteed due to the coercivity and weak lower semicontinuity of the cost functional in control variables.  
\begin{lemma} \label{lem:existence-minimizer}
	Under \crefrange{ass:domain}{ass:zeta-funcs}, problem \eqref{eq:P}   admits at least one global optimal solution.	
\end{lemma} 
\begin{proof}
	Problem \eqref{eq:P} can be written in the form 
	\begin{equation*}
		\begin{cases}
			I\left(S(z), z\right)\to \inf,\\
			z\in \Phi\subset L^p(\Omega) \times L^q(\Gamma),
		\end{cases}
	\end{equation*} 
	where $z :=(u,v) \in L^p(\Omega) \times L^q(\Gamma)$.
	Since $\zeta_i\left(0\right)=0$ and $g_i(\cdot,0)=0$ with $i=1,2$, we have $\left(0, 0\right)\in\Phi$.
	Set $\xi:=\inf_{z\in\Phi}I\left(z\right)$. From \cref{ass:L-func}, $\xi$ is finite and
	\[
		0\leq \xi \leq \int_\Omega L(x, 0)dx +\int_\Gamma \ell(x', 0)ds =:\delta<+\infty.
	\]
	By definition, there exists a sequence $\{z_k:=(u_k, v_k)\} \subset \Phi$ such that $\xi=\lim_{k\to\infty} I(S(z_k),z_k)$. There then holds that
	\begin{equation*}
		I(S(z_k), z_k) < \delta + C
	\end{equation*}
	for some constant $C>0$ and for all $k \geq 1$. 
	Hence there is a constant $R>0$ such that
	\[ 
		\|u_k\|_{L^{p}\left(\Omega\right)}+ \|v_k\|_{L^{q}\left(\Gamma\right)} \leq R,\quad \text{for all } k\geq 1.
	\] 
	Thanks to the reflexivity of $L^p(\Omega) \times L^q(\Gamma)$, there exists a subsequence, denoted in the same way, of $\{(u_k,v_k)\}$ such that
	\[
		u_k \rightharpoonup \tilde{u} \quad \text{in } L^p(\Omega) \quad \text{and} \quad v_k \rightharpoonup \tilde{v} \quad \text{in } L^q(\Gamma)
	\]
	for some $(\tilde{u},\tilde{v}) \in L^p(\Omega) \times L^q(\Gamma)$. By virtue of \cref{lem:control2state-compact}, we have
	\[
		y_k := S(u_k, v_k) \to S(\tilde{u},\tilde{v})=: \tilde{y} \quad \text{in } H^1(\Omega) \times C(\overline\Omega).
	\]
	It remains to show that
	\begin{equation}
		\label{eq:feasible-point}
		(\tilde u,\tilde v) \in \Phi,
	\end{equation}
	which together with the weak lower semicontinuity of $I(S(\cdot),\cdot)$ yields
	\[
		\xi=\lim_{k\to\infty} I\left(S(z_k), z_k\right)\geq I\left(\tilde{y}, \tilde{z}\right)
	\]
	with $z_k := (u_k, v_k)$ and $\tilde{z} := (\tilde{u}, \tilde{v})$. Thus $(\tilde{y}, \tilde{u}, \tilde{v})$ is a global optimal solution to \eqref{eq:P}.
	
	To prove \eqref{eq:feasible-point}, we first validate that
	\begin{equation} \label{eq:constraint-domain-1}
		\zeta_1 \left(\tilde u\left(x\right)\right) +g_1\left(x,\tilde y\left(x\right)\right)\leq 0 \quad \text{for a.e. } x \in \Omega.
	\end{equation}
	To this end, we exploit the fact that $z_k\in\Phi$ to get
	\begin{equation} \label{eq:constraint-domain-2}
		\zeta_1\left(u_k\left(x\right)\right) +g_1\left(x, y_k\left(x\right)\right)\leq 0 \quad \text{a.e. } x \in \Omega \quad \text{and for all } k \geq 1.
	\end{equation} 
	Since $\zeta_1$ is strictly monotone and differentiable, so is its inverse $H_1$. We now consider two cases:
	\begin{enumerate}[label=(\roman*)]
		\item \emph{$\zeta_1$ is strictly increasing.} In this case, $H_1$ is also monotonically increasing. The constraint \eqref{eq:constraint-domain-2} can be rewritten as
		\[
			\zeta_1\left(u_k\left(x\right)\right)\leq -g_1\left(x,y_k\left(x\right)\right) \quad \text{for all } k\geq 1.
		\]
		Equivalently, one has
		\[
			u_k\left(x\right) \leq H_1 [-g_1\left(x,y_k\left(x\right)\right) ] \quad \text{for all } k\geq 1.
		\]
		That means
		\[
			u_k\left(x\right) - H_1 [ -g_1\left(x,y_k\left(x\right)\right) ]\leq 0 \quad \text{for all } k\geq 1.
		\]
		Letting $k \to \infty$ and using the weak closedness of  the set $\{ u\in L^{p}\left(\Omega\right) | u\left(x\right) \leq 0 \, \text{a.e. } x \in \Omega\}$ in $L^{p}\left(\Omega\right)$, we get	
		\[
			\tilde u\left(x\right) - H_1 [ -g_1\left(x,\tilde y\left(x\right)\right) ] \leq 0.	
		\]	 
		From this, we have \eqref{eq:constraint-domain-1}.
		\item \emph{$\zeta_1$ is strictly decreasing.} In this case, $H_1$ is also monotonically decreasing.
		Similarly as (i), we also have \eqref{eq:constraint-domain-1}.
	\end{enumerate}
	The same argument as above implies that
	\begin{equation}\label{eq:constraint-boundary-1}
		g_2\left(x', \bar y\left(x'\right)\right) +\zeta_2\left(\tilde v\left(x'\right)\right)\leq 0\ {\rm a.e.}\ x'\in\Gamma. 
	\end{equation}
    The combination of \eqref{eq:constraint-domain-1} with \eqref{eq:constraint-boundary-1} yields \eqref{eq:feasible-point}.
\end{proof}

\subsection{Mathematical programming problem corresponding to \eqref{eq:P}}
In this subsection we shall first transfer problem \eqref{eq:P} to a mathematical programming problem of the form
\begin{equation}
	\label{eq:P-MP}
	\tag{P*}
	\min J(z) \quad \text{subject to } G(z) \in K,
\end{equation}
where 
the functional $J$, the mapping $G$ and the convex closed convex $K$ will be determined later. 
The corresponding optimality system at any local minimizer of \eqref{eq:P-MP} will be then derived and it will help us establish the optimality conditions for \eqref{eq:P}.

\medskip

Set 
\begin{align*}
	& Z := L^p(\Omega) \times L^q(\Gamma), \quad E:= L^p(\Omega) \times L^q(\Gamma),\\
	& K_\Omega := \left\{ u \in L^p(\Omega) \mid u \geq 0 \, \text{a.e. in } \Omega \right\},\\
	&  K_\Gamma := \left\{ v \in L^q(\Gamma) \mid v \geq 0 \, \text{a.e. in } \Gamma \right\}\\
		\intertext{and}
	& K = (K_1, K_2) := 
	\begin{cases}
		(-K_\Omega) \times (-K_\Gamma)  & \text{if $\zeta_1', \zeta_2' \geq 0$}\\
		 (-K_\Omega) \times K_\Gamma &\text{if $\zeta_1' \geq 0, \zeta_2' \leq 0$}\\
		 K_\Omega \times (-K_\Gamma) & \text{if $\zeta_1' \leq 0, \zeta_2' \geq 0$}\\
		 K_\Omega \times K_\Gamma & \text{if $\zeta_1', \zeta_2' \leq 0$}.
	\end{cases}
\end{align*}
For any $z = (u,v) \in Z$ and $y := S(z)$, we define 
\begin{align*}
	& J(z) := I(S(z),z),\\
	& G_1(z) := u - H_1 [ -g_1\left(\cdot,y\right) ],\\
	& G_2(z) := v - H_2 [ -g_2\left(\cdot,y\right) ]\\
	\intertext{and}
	& G(z) := (G_1(z), G_2(z)).
\end{align*}
Obviously, problem \eqref{eq:P} is transferred to \eqref{eq:P-MP}. Moreover, the Lagrangian associated with \eqref{eq:P-MP} (and with \eqref{eq:P}) is defined via
\begin{align}
	\mathcal{L}(z; e) & = J(z) + \langle e, G(z) \rangle_{E^*,E} \notag \\
	& = J(z) + \langle e_1, G_1(z) \rangle_{L^{p'}(\Omega), L^p(\Omega)} + \langle e_2, G_2(z) \rangle_{L^{q'}(\Gamma), L^q(\Gamma)}  \notag \\
	& = J(z) + \int_\Omega e_1\left( u - H_1 [ -g_1\left(\cdot,y\right) ] \right)dx + \int_\Gamma e_2 \left(  v - H_2 [ -g_2\left(\cdot,y\right) ]\right)d\sigma(x) \label{eq:Lagrangian-func}
\end{align}
with $z \in Z$, $e=(e_1,e_2) \in E^*$ and $y := S(z)$.
\medskip

We now arrive at the differentiability of $J$ and $G$.
\begin{lemma}
	\label{lem:diff-j-G}
	Assume that \crefrange{ass:domain}{ass:zeta-funcs} are fulfilled. Then, $J$ and $G$ are continuously differentiable at any $z \in Z$. Moreover, for any $z=(u,v), \tilde{z} = (\tilde{u}, \tilde{v}) \in Z$, there hold
	\begin{align}
		J'(z)\tilde{z} & = \int_\Omega (\varphi_{u,v} +\Delta_1(u) ) \tilde{u} dx + \int_\Gamma ( \gamma\varphi_{u,v} +\Delta_2(v) ) \tilde{v}d\sigma(x) \label{eq:der-j} \\
		\intertext{and}
		G'(z)\tilde{z} &= \left( G_1'(z) \tilde{z}, G_2'(z) \tilde{z} \right) \label{eq:der-G}.
	\end{align}
	with
	\[
		\begin{aligned}
			& \Delta_1 (t) := \lambda_1 t + \lambda_2 |t|^{p-2}t, \quad \Delta_2 (t) := \mu_1 t + \mu_2 |t|^{q-2}t,\\
			& G_1'(z) \tilde{z} = \tilde{u} + H_1'(-g_1(\cdot, y_{u,v})) g'_{1y}(\cdot, y_{u,v})  S'(z)\tilde{z},\\
			& G_2'(z) \tilde{z} =\tilde{v} + H_2'(-g_2(\cdot, y_{u,v})) g'_{2y}(\cdot, y_{u,v}) \gamma (S'(z)\tilde{z}).
		\end{aligned}
	\]
	Here $y_{u,v} := S(z)$ and $\varphi_{u,v}$  is the unique solution in $W^{1,k}(\Omega)$ for all $k \geq 1$ to
	\begin{equation}
		\label{eq:adjoint-oper}
		\left\{
		\begin{aligned}
			A^*\varphi_{u,v} + f'_y(x, y_{u,v}) \varphi_{u,v} &= L'_{y}(\cdot, y_{u,v}) \quad \text{in } \Omega,\\
			\partial_{\nu_{A^*}} \varphi_{u,v}& = l'_{y}(\cdot, y_{u,v}) \quad \text{on } \Gamma,
		\end{aligned}
		\right.
	\end{equation}
	where $A^*$ is the adjoint operator of $A$ and given by 
	\[
		A^*\varphi=-\sum_{i,j=1}^ND_i\left(a_{ij} D_j\varphi \right)+a_0 \varphi
	\]
	and
	\[ 
		\partial_{\nu_{A^*}} \varphi = \sum_{i,j=1}^N a_{ij} D_j\varphi \nu_i.
	\] 
\end{lemma}
\begin{proof}
	We first note that the right-hand sides of \eqref{eq:adjoint-oper} belong to $L^\infty(\Omega)$ and $L^\infty(\Gamma)$ because of $y_{u,v} \in C(\overline\Omega)$ and \cref{ass:L-func}. From this and the Sobolev embedding \cite{Adams,Kufner1977}, we have
	\[
		L'_{y}(\cdot, y_{u,v}) \in  \left(W^{1,k'}(\Omega) \right)^* \quad \text{and} \quad l'_{y}(\cdot, y_{u,v}) \in W^{-\frac{1}{k},k}(\Gamma)
	\]
	for all $k \geq 1$. Applying \cite[Lem.~2.4]{CasasMateos2002}, we then deduce that \eqref{eq:adjoint-oper} admits unique solution $\varphi_{u,v}$ in $W^{1,k}(\Omega)$ for all $k\geq 1$.
	Since the mapping $I_\Omega: L^p(\Omega) \ni u \mapsto \frac{\lambda_1}{2}\norm{u}^2_{L^2(\Omega)} + \frac{\lambda_2}{p} \norm{u}_{L^p(\Omega)}^p \in \R$ with $p\geq 2$ is of class $C^1$. Moreover, by virtue of \cite[Prop.~4.9]{Cioranescu1990}, we have 
	\[
		I_\Omega'(u) = \lambda_1 u + \lambda_2 |u|^{p-2}u.
	\]
	Similarly, the mapping $I_\Gamma: L^q(\Gamma) \ni v \mapsto \frac{\mu_1}{2}\norm{v}^2_{L^2(\Gamma)} + \frac{\mu_2}{p} \norm{v}_{L^q(\Gamma)}^q \in \R$ with $q \geq 2$ is of class $C^1$ and satisfies
	\[
		I_\Gamma'(v) = \mu_1 v + \mu_2 |v|^{q-2}v.
	\]
	Besides, the mapping $I_0 : C(\overline\Omega) \ni y \mapsto \int_\Omega L(x,y(x))dx + \int_\Gamma l(x,y(x))d\sigma(x) \in \R$ is continuously differentiable due to \cref{ass:L-func}. It is noted for any $(y,u,v) \in C(\overline\Omega) \times L^p(\Omega) \times L^q(\Gamma)$ that
	\[
		I(y,u,v) = I_0(y) + I_\Omega(u) + I_\Gamma(v).
	\]
	Therefore, functional $I$ is also continuously differentiable as a function on $C(\overline\Omega) \times L^p(\Omega) \times L^q(\Gamma)$. Applying the chain rule to $J(z) = I(S(z),z)$ with $z = (u,v)$ and using \cref{lem:control2state-diff}, we get that $J$ is of class $C^1$ and we derive \eqref{eq:der-j}.
	
	To prove \eqref{eq:der-G}, we first show that $H_i$ are globally Lipschitz continuous and that, for any $M>0$, there exist $K_{H_i,M}>0$ satisfying 
	\begin{equation} \label{eq:H-der-Lipschitz}
		\left|H_{i}'(\tau_1) - H_i'(\tau_2) \right| \leq K_{H_i,M}|\tau_1-\tau_2|
	\end{equation}
	for all $\tau_1, \tau_2 \in \R$ with $|\tau_1|, |\tau_2| \leq M$ for $i=1,2$.
	To this end, since $|\zeta_i'\left(\cdot\right)| \geq \rho_i>0$, its inverse $H_i$ satisfies
	\begin{equation}
		\label{eq:estiH}
		|H_i' \left(q\right)|=\frac{1}{|\zeta_i'\left(H_1\left(q\right)\right)|} \leq \frac{1}{\rho_i},
	\end{equation}
	which implies that $H_i$ is globally Lipschitz continuous. This and the fact $H_i(0)=0$ imply that $|H_i\left(\tau\right)| \leq C_M$ whenever $|\tau|\leq M$. 
	Moreover, we deduce from \eqref{eq:estiH} and \cref{ass:zeta-funcs} for $|\tau_1|, |\tau_2|\leq M$ that
	\begin{align*}
		|H_i'\left(\tau_1\right) -H_i'\left(\tau_2\right)| =& \left| \dfrac{1}{\zeta_i' \left(H_i\left(\tau_1\right)\right)} -\dfrac{1}{\zeta_i' \left(H_i\left(\tau_2\right)\right)}\right|\\
		\leq &\dfrac{1}{\rho_i^2} \left| \zeta_i' \left(H_i\left(\tau_1\right)\right) - \zeta_i' \left(H_i\left(\tau_2\right)\right)\right| \\
		\leq &\dfrac{1}{\rho_i^2} K_{\zeta_i,C_M} \left| H_i \left(\tau_1\right) - H_i\left(\tau_2\right)\right| \\
		\leq &\dfrac{1}{\rho_i^3} K_{\zeta_i,C_M} \left| \tau_1- \tau_2\right|. 
	\end{align*}
	We then derive \eqref{eq:H-der-Lipschitz}.

	Finally, from the chain rule, \cref{lem:control2state-diff}, the definition of $G$, the global Lipschitz continuity of $H_i$, $i=1,2$, \eqref{eq:H-der-Lipschitz}, and \cref{ass:g-funcs}, we derive the continuous differentiability of $G$ and formula \eqref{eq:der-G}.
\end{proof}

To derive the optimality system for problem \eqref{eq:P}, we need the following result.
\begin{proposition}
	\label{prop:Robinson}
	Problem \eqref{eq:P-MP} satisfies  Robinson's constraint qualification at any point $z \in Z$, that is,
	\begin{equation*}
		0 \in \textup{ int }\left\{G'(z)Z - (K - G(z))  \right\}.
	\end{equation*}
\end{proposition}
\begin{proof}
	According to \cite[Thm.~2.1]{ZoweKurcyusz1979} (also \cite[Lem.~2.3]{MaurerZowe1979}), it suffices to prove that
	\begin{equation}
		\label{eq:Robinson}
		E = G'(z)Z - \cone(K - G(z)),
	\end{equation}
	where
	\[
	\cone(K- \bar k):= \left\{ \lambda(k - \bar k) \mid \lambda \geq 0, k \in K \right\} \quad (\bar k \in K).
	\]
	To this end, let $z_0 =(u_0,v_0) \in E$ be arbitrary and consider the following equation
	\begin{equation}
		\label{eq:Robinson-2}
		\left\{
		\begin{aligned}
			Aw + f'_y(x, y_{u,v})w  + \frac{g'_{1y}(\cdot, y_{u,v})}{\xi_1'(H_1(-g_1(\cdot,y_{u,v})))}w &= u_0 \quad \text{in } \Omega,\\
			\partial_{\nu _A} w  + \frac{g'_{2y}(\cdot, y_{u,v})}{\xi_2'(H_2(-g_2(\cdot,y_{u,v})))}w & = v_0 \quad \text{on } \Gamma
		\end{aligned}
		\right.
	\end{equation}
	with $y_{u,v} = S(z)$ and $(u,v)=z$. By \eqref{eq:Robinson-needed}, equation \eqref{eq:Robinson-2} admits a unique solution $w \in H^1(\Omega) \cap C(\overline\Omega)$. Setting
	\begin{equation} \label{eq:Robinson-3}
		\left\{
		\begin{aligned}
			\tilde{u} &:= u_0 - \frac{g'_{1y}(\cdot, y_{u,v})}{\xi_1'(H_1(-g_1(\cdot,y_{u,v})))}w \quad \text{in } \Omega,\\
			\tilde{v} &:= v_0 - \frac{g'_{2y}(\cdot, y_{u,v})}{\xi_2'(H_2(-g_2(\cdot,y_{u,v})))}w \quad \text{on } \Gamma
		\end{aligned}
		\right.
	\end{equation}
	yields that $(\tilde{u}, \tilde{v}) \in Z$ and $w$ solves the equation
	\begin{equation*}
		\left\{
		\begin{aligned}
			Aw + f'_y(x, y_{u,v})w &= \tilde{u} \quad \text{in } \Omega,\\
			\partial_{\nu _A} w & = \tilde{v} \quad \text{on } \Gamma.
		\end{aligned}
		\right.
	\end{equation*}
	The last system gives that $w = S'(z)\tilde{z}$ with $\tilde{z} := (\tilde{u}, \tilde{v})$. From \eqref{eq:Robinson-3} and the definition of functions $H_i$ ($i=1,2$), we have
	\[
		\left\{
		\begin{aligned}
			u_0 &= \tilde{u} + H_1'(-g_1(\cdot,y_{u,v}))g'_{1y}(\cdot, y_{u,v})S'(z)\tilde{z},\\
			v_0 &= \tilde{v} + H_2'(-g_2(\cdot,y_{u,v}))g'_{2y}(\cdot, y_{u,v})\gamma(S'(z)\tilde{z}).
		\end{aligned}
		\right.
	\]
	Combining this with \eqref{eq:der-G} yields
	\[
		z_0 = G'(z)\tilde{z},
	\]
	which finally implies \eqref{eq:Robinson}.
\end{proof}

\subsection{Optimality conditions for \eqref{eq:P}} \label{sec:OC}

Let us set
\begin{equation} \label{eq:exponent-s}
	s := 
	\begin{cases}
		 \min\left\{ \left( 1- \frac{1}{p} - \frac{1}{N} \right)^{-1}, \frac{Nq}{(N-1)(q-1)} \right \} &\quad \text{if } 1- \frac{1}{p} - \frac{1}{N}  >0,\\
		 \frac{Nq}{(N-1)(q-1)}& \quad \text{otherwise}.
	\end{cases}
\end{equation}
\begin{proposition} \label{prop:conjugate}
	Let $r$ and $s$ be numbers defined in \eqref{eq:exponent-r} and \eqref{eq:exponent-s}, respectively. Then, there holds
	\begin{equation} \label{eq:conjugate}
		\frac{1}{r} + \frac{1}{s} < 1.
	\end{equation}
\end{proposition}
\begin{proof}
	By definition, we have
	\begin{equation*} 
		\frac{1}{r} = \begin{cases}
			\max\left\{ \frac{1}{p} - \frac{1}{N}, \frac{1}{q} - \frac{1}{Nq}  \right \} &\quad \text{if } \frac{N}{2} < p < N,\\
			 \frac{1}{q} - \frac{1}{Nq} & \quad \text{if } p \geq N
		\end{cases}
	\end{equation*}
	and
	\[
		 \frac{1}{s} = \begin{cases}
		 	\max\left\{  1- \frac{1}{p} - \frac{1}{N} , (1- \frac{1}{N})(1- \frac{1}{q}) \right \} &\quad \text{if } 1- \frac{1}{p} - \frac{1}{N}  >0,\\
		 	(1- \frac{1}{N})(1- \frac{1}{q})& \quad \text{otherwise}.
		 \end{cases}
	\]
	By using elementary calculations, we get
	\begin{align*}
		& \left(\frac{1}{p} - \frac{1}{N} \right) + \left(  1- \frac{1}{p} - \frac{1}{N}\right) = 1 - \frac{2}{N} < 1,\\
		&  \left(\frac{1}{p} - \frac{1}{N} \right) +  \left(1- \frac{1}{N}\right)\left(1- \frac{1}{q}\right) = 1 + \left(\frac{1}{p} - \frac{2}{N} \right) - \frac{1}{q}\left( 1 - \frac{1}{N} \right) < 1 \, \text{since } p > N/2,\\
		&\left( \frac{1}{q} - \frac{1}{Nq} \right) + \left(  1- \frac{1}{p} - \frac{1}{N}\right)= 1 - \frac{1}{p} - \frac{1}{N}\left( 1 - \frac{N-1}{q} \right) < 1 \, \text{since } q > N-1,
		\intertext{and}
		& \left( \frac{1}{q} - \frac{1}{Nq} \right) +\left(1- \frac{1}{N}\right)\left(1- \frac{1}{q}\right) = 1 - \frac{1}{N} < 1.
	\end{align*}
	These therefore give \eqref{eq:conjugate}.
\end{proof}

We now have everything in hand to derive the optimality conditions at any local minimizer of problem \eqref{eq:P}.
\begin{theorem}\label{thm:OC}
	Suppose that \crefrange{ass:domain}{ass:zeta-funcs} are fulfilled and $\left(\bar y, \bar u,\bar v\right)$ is a local minimizer of problem \eqref{eq:P}.  Then there exist multipliers $\bar \phi \in W^{1,s}(\Omega)$ with $s$ defined in \eqref{eq:exponent-s}, $\psi_1\in L^{p'}\left(\Omega\right)$ and $\psi_2\in L^{q'}\left(\Gamma\right)$ such that the following conditions are fulfilled:
	\begin{enumerate}[label=(\roman*)]
		\item (the adjoint equation)
		\begin{equation}\label{eq:adjoint-state}
		\begin{cases}
			A^*\bar\phi+f_y'(\cdot, \bar y)\bar\phi= L_y'(\cdot, \bar y)+ g_{1y}'(\cdot, \bar y)\psi_1 \quad &\text{\rm in}\quad \Omega,\\
			\partial_{\nu_{A^*}}\bar\phi = \ell_y'(\cdot, \bar y)+ g_{2y}'(\cdot, \bar y)\psi_2  \quad &\text{\rm on}\quad \Gamma;
		\end{cases}
		\end{equation}
		\item (the stationary conditions in $u$ and $v$)
		\begin{align}
			&\lambda_1\bar u + \lambda_2|\bar u|^{p-2}\bar u + \bar \phi + \zeta_1'\left(\bar u\right)\psi_1=0, \quad \text{\rm in}\ \Omega, \label{eq:stationary-u}\\
			&\mu_1\bar v +\mu_2|\bar v|^{q-2} \bar v+ \gamma \bar \phi +\zeta_2'\left(\bar v\right)\psi_2=0, \quad \text{\rm on}\ \Gamma; \label{eq:stationary-v}
		\end{align} 
		\item (the complementary conditions)
		\begin{align} 
			& \psi_1 \left(x\right)\left(\zeta_1\left(\bar u\left(x\right)\right)+g_1(x,\bar y(x)\right)=0, \quad\text{a.e.}\quad x\in \Omega, \label{eq:complementary-u}\\
			& \psi_2 \left(x'\right)\left(\zeta_2\left(\bar v\left(x'\right)\right)+g_2(x',\bar y(x')\right)=0, \quad\text{a.e.}\quad x'\in \Gamma.	 \label{eq:complementary-v}
		\end{align}
	\end{enumerate}
\end{theorem}
\begin{proof}
	Since $\left(\bar y, \bar u,\bar v\right)$ is a local minimizer of problem \eqref{eq:P}, $\bar z:= (\bar u, \bar v)$ is a local optimal solution to \eqref{eq:P-MP}.
	Thanks to \cref{prop:Robinson} and \cite[Thm.~3.9]{BonnansShapiro2000}, there exists a multiplier $e =(e_1, e_2) \in E^*$ such that
	\begin{align}
		& \mathcal{L}'_z(\bar z;e) = 0 \label{eq:Lagrangian-stationary}\\
		\intertext{and}
		& e \in N(K; G(\bar z)), \label{eq:normal-condition}
	\end{align}
	where the Lagrangian functional $\mathcal{L}$ is defined in \eqref{eq:Lagrangian-func} and
	\[
		N(K; G(\bar z)) := \left\{ e \in E^* \mid \langle e, k- G(\bar z) \rangle_{E^*,E} \leq 0 \, \text{for all } k \in K  \right\}.
	\]
	From \eqref{eq:normal-condition}, we have for $i=1,2$ that
	\begin{equation} \label{eq:normal-condition-i}
		e_i \in N(K_i; G_i(\bar z)). 
	\end{equation}
	Since $K_i$, $i=1,2$, are cones with vertex at the origin, \eqref{eq:normal-condition-i} is equivalent to
	\begin{equation}
		\left\{
			\begin{aligned}
				 & \langle e_i, G_i(\bar z) \rangle_{E_i^*, E_i} = 0, \\
				 & \langle e_i, k_i \rangle_{E_i^*, E_i} \leq 0 \quad \text{for all } k_i \in K_i, \\
				 & G_i(\bar z) \in K_i
			\end{aligned}
		\right.
	\end{equation}
	with $E_1:= L^p(\Omega)$ and $E_2:=L^{q}(\Gamma)$. From this and the definition of $K_i$, we have that
	\begin{equation*}
		\left\{
		\begin{aligned}
			& e_1(x)G_1(x) = 0 \quad \text{for a.e. } x \in \Omega,\\
			& e_2(x')G_2(x') = 0 \quad \text{for a.e. } x' \in \Gamma,
		\end{aligned}
		\right.
	\end{equation*}
	which infers that
	\begin{equation*}
		\left\{
		\begin{aligned}
			& e_1(x)(\bar u(x) - H_1(-g_1(x,\bar y(x))) = 0 \quad \text{for a.e. } x \in \Omega,\\
			& e_2(x')(\bar v(x') - H_2(-g_2(x',\bar y(x'))) = 0 \quad \text{for a.e. } x' \in \Gamma.
		\end{aligned}
		\right.
	\end{equation*}
	The definition of $H_i$ then implies that
	\begin{equation} \label{eq:complement-uv}
		\left\{
		\begin{aligned}
			& e_1(x)\left(\zeta_1\left(\bar u\left(x\right)\right)+g_1(x,\bar y(x))\right)=0 \quad \text{for a.e. } x \in \Omega,\\
			& e_2(x')\left(\zeta_2\left(\bar v\left(x'\right)\right)+g_2(x',\bar y(x'))\right)=0 \quad \text{for a.e. } x' \in \Gamma.
		\end{aligned}
		\right.
	\end{equation}
	On the other hand, by setting $\psi_1 := \frac{e_1}{\zeta_1'(\bar u)}$ and $\psi_2 := \frac{e_2}{\zeta_2'(\bar v)}$, we deduce from \eqref{eq:complement-uv} 
	that \eqref{eq:complementary-u} and \eqref{eq:complementary-v} are verified. Also, \eqref{eq:complement-uv} and the definition of $H_i$, $i=1,2$ indicate that 
	\begin{equation}
		\label{eq:ei-psi}
		\begin{aligned}[b]
		H_i'(-g_i(\cdot, \bar y))e_i &=\frac{1}{\zeta_i'(H_i(-g_i(\cdot, \bar y)))}e_i \\
		& =\frac{1}{\zeta_i'(H_i(\zeta_i(\bar z_i)))}e_i \\
		&  = \psi_i,
		\end{aligned}
	\end{equation}
	where $\bar z_1 := \bar u$ and $\bar z_2 := \bar v$. 
	Besides, we have from the definition of $\mathcal{L}$ and \cref{lem:diff-j-G} for any $\tilde{z} = (\tilde{u}, \tilde{v})$ that
	\begin{align}
		\mathcal{L}'_z(\bar z;e)\tilde{z}& = \int_\Omega (\varphi_{\bar u,\bar v} +\Delta_1(\bar u) ) \tilde{u} dx + \int_\Gamma ( \gamma \varphi_{\bar u,\bar v} +\Delta_2(\bar v) ) \tilde{v}d\sigma(x) \notag \\
		& \qquad + \int_\Omega e_1 \left[ \tilde{u} + H_1'(-g_1(\cdot, \bar y)) g'_{1y}(\cdot, \bar y) w  \right]dx \notag \\
		& \qquad + \int_\Gamma e_2 \left[ \tilde{v} + H_2'(-g_2(\cdot, \bar y)) g'_{2y}(\cdot, \bar y) \gamma w \right]d\sigma(x) \label{eq:Lagrangian-der-1}
	\end{align}
	with $w := S'(\bar z)\tilde{z}$. Here $\varphi_{\bar u,\bar v}$ denotes the unique solution to \eqref{eq:adjoint-oper} corresponding to $u := \bar u, v := \bar v$. Now consider the following  equation
	\begin{equation}
		\label{eq:adjoint-oper-chi}
		\left\{
		\begin{aligned}
			A^*\chi + f'_y(x, \bar y) \chi &= H_1'(-g_1(\cdot, \bar y)) g'_{1y}(\cdot, \bar y)e_1 \quad \text{in } \Omega,\\
			\partial_{\nu_{A^*}} \chi& = H_2'(-g_2(\cdot, \bar y)) g'_{2y}(\cdot, \bar y)e_2 \quad \text{on } \Gamma.
		\end{aligned}
		\right.
	\end{equation}
	Since $e_1 \in L^{p'}(\Omega)$, the first equation in \eqref{eq:adjoint-oper-chi} has the right-hand side term belonging to $L^{p'}(\Omega)$. From this and the Sobolev embedding, we have 
	\begin{equation*}
		H_1'(-g_1(\cdot, \bar y)) g'_{1y}(\cdot, \bar y)e_1 \in \left(W^{1,s_0'}(\Omega) \right)^*
	\end{equation*}
	with
	\begin{equation*}
		s_0 \begin{cases}
			 =\left( 1- \frac{1}{p} - \frac{1}{N} \right)^{-1} &\quad \text{if } 1- \frac{1}{p} - \frac{1}{N}  >0,\\
			\geq 1& \quad \text{otherwise}.
		\end{cases}
	\end{equation*}
	Analogously, there holds
	\[
		H_2'(-g_2(\cdot, \bar y)) g'_{2y}(\cdot, \bar y)e_2 \in W^{-\frac{1}{k},k}(\Gamma)
	\]
	with $k := \frac{N}{N-1}\frac{q}{q-1}$. In view of \cite[Lem.~2.4]{CasasMateos2002}, equation \eqref{eq:adjoint-oper-chi} has a unique solution $\chi$ in $W^{1,s}(\Omega)$ for $s$ determined in \eqref{eq:exponent-s}.
	
	On the other hand, \cref{lem:control2state-diff} implies that $w \in W^{1,r}(\Omega)$ with $r$ defined in \eqref{eq:exponent-r}. Thanks to \eqref{eq:conjugate}, we have
	\begin{equation*}
		w \in W^{1,s'}(\Omega) \quad \text{and} \quad \chi \in W^{1,r'}(\Omega).
	\end{equation*}
	Testing \eqref{eq:diff-S} and \eqref{eq:adjoint-oper-chi} by $\chi$ and $w$, respectively, and then subtracting the obtained results, we get
	\begin{multline*}
		\int_\Omega e_1  H_1'(-g_1(\cdot, \bar y)) g'_{1y}(\cdot, \bar y) w  dx + \int_\Gamma e_2   H_2'(-g_2(\cdot, \bar y)) g'_{2y}(\cdot, \bar y) \gamma w d\sigma(x) \\
		= \int_\Omega \tilde{u}\chi dx + \int_\Gamma \tilde{v}\gamma \chi d\sigma(x),
	\end{multline*}
	which, along with \eqref{eq:Lagrangian-der-1}, yields
	\begin{equation} \label{eq:Lagrangian-der-2}
			\mathcal{L}'_z(\bar z;e)\tilde{z}= \int_\Omega (\varphi_{\bar u,\bar v} + \chi + e_1 +\Delta_1(\bar u) ) \tilde{u} dx + \int_\Gamma ( \gamma\varphi_{\bar u,\bar v} + \gamma\chi + e_2 +\Delta_2(\bar v) ) \tilde{v}d\sigma(x).
	\end{equation}
	Putting $\bar \phi := \varphi_{\bar u,\bar v} + \chi$, 
	we then conclude from \eqref{eq:adjoint-oper}, \eqref{eq:adjoint-oper-chi}, \eqref{eq:ei-psi} that $\bar \phi$ satisfies \eqref{eq:adjoint-state}.
	Besides, \eqref{eq:Lagrangian-der-2} implies that	
	\begin{equation} \label{eq:Lagrangian-der-3}
		\mathcal{L}'_z(\bar z;e)\tilde{z} = \int_\Omega (\bar\phi + \zeta_1'(\bar u)\psi_1 +\Delta_1(\bar u) ) \tilde{u} dx + \int_\Gamma ( \gamma\bar\phi + \zeta_2'(\bar v)\psi_2 +\Delta_2(\bar v) ) \tilde{v}d\sigma(x).
	\end{equation}
	Combining \eqref{eq:Lagrangian-stationary} with \eqref{eq:Lagrangian-der-3} gives \eqref{eq:stationary-u} and \eqref{eq:stationary-v}. 	
\end{proof}

\section{Lipschitz continuity of optimal solutions and corresponding multipliers} \label{sec:Lipschitz-reg}

In this section, we shall prove our main result on Lipschitz continuity property of the optimal solutions and the corresponding Lagrange multipliers.

\begin{theorem}\label{thm:Lipschitz}
	Suppose that \crefrange{ass:domain}{ass:zeta-funcs} are satisfied. Furthermore assume for any $M>0$ that there exist constants $C_{L,M}$ and $C_{\ell,M}$ satisfying
	\begin{align}
		&|L_y' \left(x_1,y_1\right)-L_y'\left(x_2,y_2\right)| \leq C_{L,M}(|x_1-x_2| +|y_1-y_2|), \label{ass:Lip-L}\\
		&|\ell_y' \left(x_1',y_1\right)-\ell_y'\left(x_2',y_2\right)| \leq C_{\ell, M}(|x_1'-x_2'| +|y_1-y_2|) \label{ass:Lip-l}
	\end{align}
	for all $x_1,x_2 \in \overline\Omega$, $x_1',x_2' \in \Gamma$, and all $y_1, y_2 \in \R$ with $|y_i| \leq M$, $i=1,2$.
	Then $\bar y, \bar u, \bar \phi, \psi_1$ and $\bar v, \psi_2$ determined in \cref{thm:OC} are Lipschitz continuous on $\overline\Omega$ and $\Gamma$, respectively. 
	\end{theorem}
 
To prove the Lipschitz continuity regularity of minimizers and the corresponding multipliers, we need the following lemmas.
		
\begin{lemma}
	\label{lem:Holder-regularity}
	Under all assumptions of \cref{thm:Lipschitz}, there hold
	\begin{equation}
		\label{eq:W1k-regularity}
		\left\{
			\begin{aligned}
				\bar y, \bar \phi \in W^{1,k}(\Omega),\\   
				\bar u, \psi_1 \in L^\infty(\Omega),\\
				\bar v, \psi_2 \in L^\infty(\Gamma)
			\end{aligned}
		\right.
	\end{equation}
	for all $k \geq 1$. 
\end{lemma}		
\begin{proof}
	By setting
	\begin{align*}
		& \Omega_0 := \left\{x \in \Omega \mid \zeta_1(\bar u(x)) + g_1(x, \bar y(x)) = 0 \right\}
	\end{align*}
	we deduce from \eqref{eq:stationary-u} and \eqref{eq:complementary-u} that
	\begin{align}
		\psi_1(x) & = \begin{cases}
		-\frac{1}{\zeta_1'(\bar u(x))} \left[ \bar \phi(x)+ \Delta_1(\bar u(x)) \right]& \text{a.e. } x \in \Omega_0\\
		0 & \text{otherwise} 
		\end{cases} \notag \\
		& = \begin{cases}
		-\frac{1}{\zeta_1'(\bar u(x))} \left[ \bar \phi(x)+ \Delta_1(H_1(-g_1(x,\bar y(x)))) \right]& \text{a.e. } x \in \Omega_0\\
		0 & \text{otherwise} . \label{eq:psi1-expression}
		\end{cases}
	\end{align}
	Similarly, one has from \eqref{eq:stationary-v} and \eqref{eq:complementary-v} that
	\begin{align}
		\psi_2(x) 
		& = \begin{cases}
		-\frac{1}{\zeta_2'(\bar v(x))} \left[ \gamma\bar \phi(x)+ \Delta_2(H_2(-g_2(x,\bar y(x)))) \right]& \text{a.e. } x \in \Gamma_0\\
		0 & \text{a.e. } x \in \Gamma \backslash \Gamma_0, \label{eq:psi2-expression}
		\end{cases}
	\end{align}
	where
	\[
		\Gamma_0 := \left\{x \in \Gamma \mid \zeta_2(\bar v(x)) + g_2(x, \bar y(x)) = 0 \right\}.
	\]
	From Sobolev embeddings \cite{Adams,Kufner1977}, we have
	\begin{equation}
		\label{eq:Sobolev-embedding}
		\left\{
		\begin{aligned}
			&W^{1,t_0}(\Omega) \hookrightarrow \left(W^{1,t_1'}(\Omega) \right)^*\\
			& \gamma\left(W^{1,t_0}(\Omega) \right) \hookrightarrow W^{-\frac{1}{t_1},t_1}(\Gamma)
		\end{aligned}  
		\right. \quad \text{for} \quad 
		\left[
			\begin{aligned}
				 &1 \leq t_0 < N, \quad \frac{1}{t_1} = \frac{1}{t_0}- \frac{1}{N},\\
				 &  t_0 \geq N, \quad t_1 \geq 1 \text{ arbitrary}.
			\end{aligned}
		\right.
	\end{equation}
	It then follows from \eqref{eq:psi1-expression}, \eqref{eq:psi2-expression}, \eqref{eq:Sobolev-embedding} and the fact $\bar y \in C(\overline\Omega)$ and $\bar \phi \in W^{1,s}(\Omega)$ that 
	\[
		\psi_1 \in \left(W^{1,k_1'}(\Omega) \right)^* \quad \text{and} \quad \psi_2 \in W^{-\frac{1}{k_1},k_1}(\Gamma)
	\]
	with $k_1$ satisfying
	\[
		\left\{
		\begin{aligned}
			k_1 \geq 1 \, \text{arbitrary},  \quad &  \text{if }  s \geq N,\\
			\frac{1}{k_1} = \frac{1}{k_0} - \frac{1}{N},\quad &  \text{if } k_0 := s \in(1,N).
		\end{aligned}
		\right.
	\]
	The right hand-side terms of \eqref{eq:adjoint-state} corresponding to $\psi_1$ and $\psi_2$, respectively, therefore belong to $ \left(W^{1,k_1'}(\Omega) \right)^*$ and $W^{-\frac{1}{k_1},k_1}(\Gamma)$ because of \crefrange{ass:L-func}{ass:zeta-funcs} and of the fact that $\bar y \in C(\overline\Omega)$. 
	From this and \cite[Lem.~2.4]{CasasMateos2002}, there holds $\bar \phi \in W^{1,k_1}(\Omega)$. Using \eqref{eq:psi1-expression}, \eqref{eq:psi2-expression} and \eqref{eq:Sobolev-embedding} again, we have
	\[
		\psi_1 \in \left(W^{1,k_2'}(\Omega) \right)^* \quad \text{and} \quad \psi_2 \in W^{-\frac{1}{k_2},k_2}(\Gamma)
	\]
	for $\frac{1}{k_2} = \frac{1}{k_1} - \frac{1}{N} = \frac{1}{k_0} - \frac{2}{N}$ if $k_1 <N$ and for $k_2 \geq 1$ arbitrary if $k_1 \geq N$. 
	\cite[Lem.~2.4]{CasasMateos2002} repeatedly implies that $\bar\phi \in W^{1,k_2}(\Omega)$. The bootstrapping argument finally gives
	
	\[
		\bar\phi \in W^{1,k_m}(\Omega) \hookrightarrow C(\overline{\Omega})
	\]
	for some $k_m > N$. This together with \eqref{eq:psi1-expression} and \eqref{eq:psi2-expression} as well as the continuity of $\bar y$ on $\overline\Omega$ yields 
	\begin{equation} \label{eq:L-infty-psi1}
		\psi_1 \in L^\infty(\Omega) \quad \text{and} \quad \psi_2 \in L^\infty(\Gamma).
	\end{equation}
	The regularity of solutions to \eqref{eq:adjoint-state} infers that 
	\begin{equation}
		\label{eq:W1k-regularity-phi}
		\bar\phi \in W^{1,k}(\Omega) \quad \text{for all } k \geq 1.
	\end{equation}
	The Sobolev embedding \cite{Adams} therefore implies that
	\begin{equation}
		\label{eq:Holder-phi} 
		\bar \phi \in C^{0,\gamma_0}(\overline\Omega) \quad \text{for any } \gamma_0 \in [0,1).
	\end{equation}
	By definition, function $\Delta_1$ is of class $C^1$ and satisfies $\Delta_1'(t) = \lambda_1 + (p-1)\lambda_2 |t|^{p-2} \geq \lambda_1 >0$ for all $t \in \R$. Its inverse is also of class $C^1$ and fulfills
	\begin{equation} \label{eq:Delta-inverse}
		0 < (\Delta_1^{-1})'(\tau) \leq \frac{1}{\lambda_1} \quad \text{for all } \tau \in \R. 
	\end{equation}
	In particular, $\Delta_1^{-1}$ is globally Lipschitz continuous. 
	By virtue of \eqref{eq:stationary-u}, we have
	\begin{equation} \label{eq:u-expression}
		\bar u = \Delta_1^{-1}\left[ -\bar\phi - \zeta_1'(\bar u)\psi_1 \right].
	\end{equation}
	Moreover, we deduce from  \eqref{eq:complementary-u} that 
	\[
		\zeta_1'(\bar u) \psi_1 = \begin{cases}
		0 & \text{a.e. in }  \Omega\backslash \Omega_0,\\
		\psi_1 \zeta_1'\left(H_1(-g_1(\cdot, \bar y)) \right) & \text{a.e. in }  \Omega_0,
		\end{cases}
	\]	
	which together with \eqref{eq:L-infty-psi1} yields $\zeta_1'(\bar u) \psi_1 \in L^\infty(\Omega)$. 
	This, \eqref{eq:Holder-phi}, \eqref{eq:u-expression}, and the global Lipschitz continuity of $\Delta_1^{-1}$ give 
	\begin{equation}
		\label{eq:L-infty-u}
		\bar u \in L^\infty(\Omega).
	\end{equation}
	Similarly, one has 
	\begin{equation}
		\label{eq:L-infty-v}
		\bar v \in L^\infty(\Gamma).
	\end{equation}
	The regularity of solutions to \eqref{eq:state} thus shows that 
	\begin{equation} \label{eq:W1k-regularity-state}
		\bar y \in W^{1,k}(\Omega) \quad \text{for all } k\geq 1.
	\end{equation}
	The conclusion \eqref{eq:W1k-regularity} finally follows from \eqref{eq:L-infty-psi1}, \eqref{eq:W1k-regularity-phi}, \eqref{eq:L-infty-u}, \eqref{eq:L-infty-v} and \eqref{eq:W1k-regularity-state}.
\end{proof} 

The next lemma shows that the Nemytskii operator corresponding to a Lipschitz function is from $W^{\tau,k}(\Gamma)$ to itself.
\begin{lemma}
	\label{lem:chain-rule}
	Assume that $\Omega$ be a bounded domain in $\R^N$ with a Lipschitz boundary $\Gamma$. 
	Let $a: \Gamma \times \R \to \R$ be Lipschitz continuous  and let $v \in W^{\tau,k}(\Gamma)$ for some $\tau \in (0,1)$ and $k\geq 1$. Then, there holds
	\begin{equation*}
		a(\cdot,v(\cdot)) \in W^{\tau,k}(\Gamma).
	\end{equation*}
	Moreover, there exists a constant $C=C(N,\tau, k, \Lip(a))$ independent of $v$ such that
	\begin{equation}
		\label{eq:chain-rule-esti}
		\norm{a(\cdot, v(\cdot))}_{W^{\tau,k}(\Gamma)} \leq C\left[ \norm{v}_{W^{\tau,k}(\Gamma)} + \norm{a(\cdot,0)}_{L^k(\Gamma)} + 1  \right].
	\end{equation}
	Here $\Lip(a)$ stands for the Lipschitz constant of $a$.
\end{lemma}
\begin{proof}
	It suffices to prove \eqref{eq:chain-rule-esti}. To this end, by definition of the norm in $W^{\tau,k}(\Gamma)$, the Lipschitz continuity of $a$, and the inequality $(a_1+a_2)^k \leq 2^{k-1}(a_1^k+a_2^k)$ for $a_1,a_2>0$, $k\geq 1$, we have
	\begin{align*}
		\norm{a(\cdot, v(\cdot))}_{W^{\tau,k}(\Gamma)}^k & = \int_\Gamma |a(x,v(x))|^k d\sigma(x) + \iint_{\Gamma \times\Gamma} \frac{|a(x,v(x))-a(x',v(x'))|^{k}}{|x-x'|^{N-1+\tau k}}d\sigma(x)d\sigma(x') \\
		& \leq \int_\Gamma \left[|a(x,v(x)) - a(x,0)| + |a(x,0)| \right]^k d\sigma(x) \\
		& \qquad + \Lip(a)^k \iint_{\Gamma \times\Gamma} \frac{\left[|x-x'| + |v(x)-v(x')| \right]^k }{|x-x'|^{N-1+\tau k}}d\sigma(x)d\sigma(x')\\
		& \leq 2^{k-1} \int_\Gamma \left[ |a(x,v(x)) - a(x,0)|^k + |a(x,0)|^k \right]d\sigma(x) +\\
		& \qquad + 2^{k-1} \Lip(a)^k \iint_{\Gamma \times\Gamma} \frac{\left[|x-x'|^k + |v(x)-v(x')|^k \right]}{|x-x'|^{N-1+\tau k}} d\sigma(x)d\sigma(x')\\
		& \leq  2^{k-1} \int_\Gamma \left[ \Lip(a)^k |v(x)|^k+ |a(x,0)|^k \right]d\sigma(x)\\
		& \qquad  + 2^{k-1} \Lip(a)^k \iint_{\Gamma \times\Gamma}  \frac{1}{|x-x'|^{N-1 +(\tau -1)k}}d\sigma(x)d\sigma(x') \\
		& \qquad + 2^{k-1} \Lip(a)^k \iint_{\Gamma \times\Gamma}   \frac{|v(x)-v(x')|^k}{|x-x'|^{N-1+\tau k}} d\sigma(x)d\sigma(x')\\
		& \leq C(k,\Lip(a)) \left[ \norm{v}^k_{W^{\tau,k}(\Gamma)} + \norm{a(\cdot,0)}_{L^k(\Gamma)}^k\right] \\
		& \qquad  +  2^{k-1} \Lip(a)^k \iint_{\Gamma \times\Gamma}  \frac{1}{|x-x'|^{N-1 +(\tau -1)k}}d\sigma(x)d\sigma(x') .
	\end{align*}
	Since $N-1 + (\tau-1)k<N-1$ and $\Gamma$ is a $(N-1)-$dimensional Lipschitz submanifold of $\R^N$, there holds that
	\begin{equation*}
		\iint_{\Gamma \times\Gamma}  \frac{1}{|x-x'|^{N-1 +(\tau -1)k}}d\sigma(x)d\sigma(x') < C
	\end{equation*}
	for some constant $C$ depending only on $(N-1)+(\tau-1)k$; see, e.g. \cite[p.~331]{Kufner1977}. We therefore derive \eqref{eq:chain-rule-esti}.
\end{proof}		
\begin{remark}
	\label{rem:chain-rule}
	Assume that $a: \Gamma \times \R \to \R$ fulfills the property that, for any $M>0$, there is a constant $C_{M,a}$ such that
	\begin{equation}
		|a(x_1,y_1)-a(x_2,y_2)| \leq C_{M,a}\left( |x_1-x_2| + |y_1-y_2|\right)
	\end{equation}
	for all $x_1, x_2 \in \Gamma$ and $y_1,y_2\in \R$ with $|y_i| \leq M$ for $i=1,2$. Then, for any $M>0$, by using similar arguments as in the proof of \cref{lem:chain-rule}, there is a constant $C=C(k,\tau, N, C_{M,a})$ satisfying
	\begin{equation}
		\label{eq:chain-rule-esti-2}
		\norm{a(\cdot, v(\cdot))}_{W^{\tau,k}(\Gamma)} \leq C \left[ \norm{v}_{W^{\tau,k}(\Gamma)} + \norm{a(\cdot,0)}_{L^k(\Gamma)} + 1  \right]
	\end{equation}
	for any $v \in W^{\tau, k}(\Gamma) \cap L^\infty(\Gamma)$ with $\tau \in (0,1)$, $k \geq 1$, and $\norm{v}_{L^\infty(\Gamma)} \leq M$.
\end{remark}

The following result indicates that the product of functions in Sobolev spaces of fractional order is also a $W^{\tau,k}(\Gamma)$-function.
\begin{lemma}
	\label{lem:production} 
	Let $\Omega$ be a bounded domain in $\R^N$ with a Lipschitz boundary $\Gamma$. Let $\tau, \tau_1, \tau_2 \in (0,1)$ and $k, k_1, k_2 \geq 1$ be such that
	\begin{equation}
		\label{eq:product-constants}
			\left\{
		\begin{aligned}
		& \frac{1}{k} = \frac{1}{k_1} + \frac{1}{k_2},\\
		& 0 < \tau < \min\{\tau_1, \tau_2 \}.
		\end{aligned}
		\right.
	\end{equation}
	Then, for any  $v_1 \in W^{\tau_1,k_1}(\Gamma)$ and $v_2 \in W^{\tau_2,k_2}(\Gamma)$, there hold that $v_1v_2 \in W^{\tau,k}(\Gamma)$ and that
	\begin{equation}
		\label{eq:product-Wk}
		\norm{v_1v_2}_{W^{\tau,k}(\Gamma)} \leq C \norm{v_1}_{W^{\tau_1,k_1}(\Gamma)} \norm{v_2}_{W^{\tau_2,k_2}(\Gamma)} 
	\end{equation}
	for some constant $C=C(N,\tau,\tau_1,\tau_2,k,k_1,k_2)$.
\end{lemma}
\begin{proof}
	It is sufficient to show \eqref{eq:product-Wk}. To prove this, by definition of the norm in $W^{\tau,k}(\Gamma)$-space, we have
	\begin{equation}
		\label{eq:prod-esti1}
		\norm{v_1v_2}_{W^{\tau,k}(\Gamma)}^k = \norm{v_1v_2}_{L^k(\Gamma)}^k + I_{\tau,k}(v_1v_2), 
	\end{equation}
	where $I_{\tau,k}$ is defined in \eqref{eq:fractional-term}. From the first relation in \eqref{eq:product-constants} and the H\"{o}lder inequality, we derive
	\begin{equation}
		\label{eq:no-frac-term}
		\norm{v_1v_2}_{L^k(\Gamma)} \leq \norm{v_1}_{L^{k_1}(\Gamma)}\norm{v_2}_{L^{k_2}(\Gamma)}.
	\end{equation}
	On the other hand, using the inequality $(a_1+a_2)^k \leq 2^{k-1}(a_1^k+a_2^k)$ with $a_1, a_2 >0$  yields
	\begin{equation}
		\label{eq:prod-esti2}
		\begin{aligned}[b]
				I_{\tau,k}(v_1v_2) & = \iint_{\Gamma \times\Gamma} \frac{|v_1(x)v_2(x)-v_1(x')v_2(x')|^k}{|x-x'|^{N-1+\tau k}}d\sigma(x)d\sigma(x') \\
			& \leq \iint_{\Gamma \times\Gamma} \frac{\left[|v_1(x)||v_2(x)-v_2(x')| + |v_2(x')||v_1(x)-v_1(x')| \right]^k}{|x-x'|^{N-1+\tau k}}d\sigma(x)d\sigma(x')\\
			& \leq 2^{k-1} \iint_{\Gamma \times\Gamma} \frac{|v_1(x)|^k|v_2(x)-v_2(x')|^k }{|x-x'|^{N-1+\tau k}}d\sigma(x)d\sigma(x') \\
			&\qquad \qquad +  2^{k-1} \iint_{\Gamma \times\Gamma} \frac{ |v_2(x')|^k|v_1(x)-v_1(x')|^k}{|x-x'|^{N-1+\tau k}}d\sigma(x)d\sigma(x') \\
			& =: 2^{k-1}A+ 2^{k-1}B.
		\end{aligned}
	\end{equation}
	Since $1 = \frac{k}{k_1} + \frac{k}{k_2}$, we can write
	\begin{align*}
		A & = \iint_{\Gamma \times\Gamma} \frac{|v_1(x)|^k|v_2(x)-v_2(x')|^k }{|x-x'|^{N-1+\tau k}}d\sigma(x)d\sigma(x')\\
		& = \iint_{\Gamma \times\Gamma} \frac{|v_1(x)|^k}{|x-x'|^{(N-1)\frac{k}{k_1}+(\tau-\tau_2)k}} \frac{|v_2(x)-v_2(x')|^k}{|x-x'|^{(N-1)\frac{k}{k_2}+\tau_2 k}} d\sigma(x)d\sigma(x')\\
		& \leq \left[ \iint_{\Gamma \times\Gamma} \frac{|v_1(x)|^{k_1}}{|x-x'|^{(N-1)+(\tau-\tau_2)k_1}} d\sigma(x)d\sigma(x') \right]^{k/{k_1}} \\
		&\qquad \qquad \times  \left[\iint_{\Gamma \times\Gamma} \frac{|v_2(x)-v_2(x')|^{k_2}}{|x-x'|^{(N-1)+\tau_2 k_2}} d\sigma(x)d\sigma(x')\right]^{k/k_2} \\
		& = \left[ \iint_{\Gamma \times\Gamma} \frac{|v_1(x)|^{k_1}}{|x-x'|^{(N-1)+(\tau-\tau_2)k_1}} d\sigma(x)d\sigma(x') \right]^{k/{k_1}} I_{\tau_2,k_2}(v_2)^k,
	\end{align*}
	where we have just used the H\"{o}lder inequality to derive the last estimate.
	Since $0<\tau < \tau_2$, there holds $N-1 + (\tau-\tau_2)k_1 < N-1$ and we  thus have from the Lipschitz continuity of the boundary $\Gamma$ that
	\[
		\int_\Gamma \frac{d\sigma(x')}{|x-x'|^{(N-1)+(\tau-\tau_2)k_1}} \leq  C(N,\tau,\tau_2,k_1) := C_1.
	\]
	for all $x \in \Gamma$. This implies that
	\[
		\iint_{\Gamma \times\Gamma} \frac{|v_1(x)|^{k_1}}{|x-x'|^{(N-1)+(\tau-\tau_2)k_1}} d\sigma(x)d\sigma(x') \leq C_1\norm{v_1}_{L^{k_1}(\Gamma)}^{k_1}.
	\]
	Hence, we have
	\[
		A \leq C_1 \norm{v_1}_{L^{k_1}(\Gamma)}^{k}I_{\tau_2,k_2}(v_2)^k.
	\]
	Analogously, one has
	\[
		B \leq C_2 \norm{v_2}_{L^{k_2}(\Gamma)}^{k}I_{\tau_1,k_1}(v_1)^k
	\]
	for some $C_2 := C(N,\tau,\tau_1,k_2)$. Therefore, \eqref{eq:prod-esti2} implies that
	\begin{equation*}
		I_{\tau,k}(v_1v_2)^{1/k}  \leq C \left(\norm{v_1}_{L^{k_1}(\Gamma)}I_{\tau_2,k_2}(v_2) + \norm{v_2}_{L^{k_2}(\Gamma)}I_{\tau_1,k_1}(v_1) \right)
	\end{equation*}
	for some $C=C(N,\tau,\tau_1,\tau_2,k,k_1,k_2)$. Combing this with \eqref{eq:prod-esti1} and \eqref{eq:no-frac-term}, we arrive at the desired conclusion.
\end{proof}

\medskip

\noindent\textbf{Proof of \cref{thm:Lipschitz}}.
If suffices to consider the case where $\zeta_1$ and $\zeta_2$ are both monotonically increasing since the other cases can be dealt with analogously. 
Assume that $\zeta_1$ and $\zeta_2$ are both monotonically increasing. We then have
\begin{align*}
	& \zeta_1'(t) \geq \rho_1 >0, \quad \zeta_2'(\tau) \geq \rho_2 >0,\\
	& K_1 = - K_\Omega = \left\{ u \in L^p(\Omega) \mid u \leq 0 \, \text{a.e. in } \Omega \right\}
	\intertext{and}
	& K_2 =- K_\Gamma = \left\{ v \in L^q(\Gamma) \mid v \leq 0\, \text{a.e. in } \Gamma \right\}.
\end{align*}
Let $e_1$ and $e_2$ be the functions determined in the proof of \cref{thm:OC}. From \eqref{eq:normal-condition-i} and \cite[Lem.~2.4]{KienNhuSon2017} (see also \cite[Lem.~4.11]{BayenBonnansSilva2013}), we conclude that
\[
	e_2(x) \in N((-\infty, 0]; \bar v(x) - H_2(- g_2(x,\bar y(x)))) \quad \text{for a.e. } x \in \Gamma,
\]
where
\[
	N((-\infty, 0]; t) := \left\{\tau \in \R \mid \tau(m-t) \leq 0 \, \text{for all } m \leq 0  \right\}.
\]
Consequently, one has
\[
	e_2(x)\left(m -\bar v(x) + H_2( -g_2(x,\bar y(x))) \right) \leq 0 \quad \text{for a.e. } x \in \Gamma \quad \text{and for all } m \leq 0.
\]
Since $\psi_2 =\frac{e_2}{\zeta_2'(\bar v)}$ and $\zeta_2'(\cdot) \geq \rho_2 >0$, there holds
\[
	\zeta_2'(\bar v(x))\psi_2(x)\left(m - \bar v(x) + H_2(-g_2(x,\bar y(x))) \right) \leq 0 \quad \text{for a.e. } x \in \Gamma \quad \text{and for all } m \leq 0.
\]
Combining this with \eqref{eq:stationary-v} yields
\[
	\left[ \Delta_2(\bar v(x)) + \bar\phi(x) \right]\left(m - \bar v(x) + H_2(-g_2(x,\bar y(x))) \right) \geq 0 \quad \text{for a.e. } x \in \Gamma \quad \text{and for all } m \leq 0.
\]	
It is noted that we have used the fact that $\gamma\bar\phi = \bar\phi$ on $\Gamma$ due to the H\"{o}lder continuity \eqref{eq:Holder-phi} of $\bar\phi$.
Since $\Delta_2'(\tau) \geq \mu_1>0$ for all $\tau \in \R$, $\Delta_2^{-1}$ exists and is globally Lipschitz continuous. 
Setting $\bar \omega_2 := \Delta_2^{-1}(-\bar\phi)$, it follows from \eqref{eq:W1k-regularity} that
\begin{equation} \label{eq:W1k-regularity-omega}
	\bar\omega_2 \in W^{1,k}(\Omega) \quad \text{for all } k \geq 1.
\end{equation}
Furthermore, $\bar\omega_2$ satisfies
\[
	\left[ \Delta_2(\bar v(x)) - \Delta_2(\bar\omega_2(x)) \right]\left(m - \bar v(x) + H_2(-g_2(x,\bar y(x))) \right) \geq 0
\]
for a.e.  $x \in \Gamma$ and for all $m \leq 0$.
The strictly increasing monotonicity of $\Delta_2$ then gives
\begin{equation*} 
	\left( \bar v(x) - \bar \omega_2(x)\right) \left(m - \bar v(x) + H_2(-g_2(x,\bar y(x))) \right) \geq 0 \quad \text{for a.e. } x \in \Gamma \, \text{and for all } m \leq 0,
\end{equation*}
or equivalently,
\begin{multline*}
	\left[  \bar \omega_2(x) - H_2(-g_2(x,\bar y(x))) - \left(\bar v(x)- H_2(-g_2(x,\bar y(x)))\right) \right] \\
	 \times \left[ m- \left(\bar v(x) - H_2(-g_2(x,\bar y(x))) \right)\right] \leq 0
\end{multline*}
for a.e. $x \in \Gamma$ and for all $m \leq 0$. This and the projection onto a closed convex set \cite[Thm.~5.2]{Brezis2010} imply that
\begin{equation}
	\label{eq:projection-v}
	\bar v(x) = \proj_{(-\infty, 0]}\left[ \bar \omega_2(x) - H_2(-g_2(x,\bar y(x))) \right] + H_2(-g_2(x,\bar y(x))) \quad \text{for a.e. } x \in \Gamma.
\end{equation}
Here $\proj_{(-\infty, 0]}(\cdot)$ denotes the projection mapping from $\R$ onto $(-\infty,0]$. 
Moreover, the global Lipschitz continuity of $H_2$ (see the proof of \cref{lem:diff-j-G}), \eqref{eq:W1k-regularity-state}, and \cref{lem:chain-rule} imply that
$H_2(-g_2(\cdot,\bar y(\cdot))) \in W^{\frac{1}{k'}, k}(\Gamma)$ for all $k \geq 1$. Combing this with \eqref{eq:W1k-regularity-omega}, \eqref{eq:projection-v} and the Lipschitz continuity of the projection mapping yields
\begin{equation}
	\label{eq:Wk-regularity-v}
	\bar v \in W^{\frac{1}{k'}, k}(\Gamma) \quad \text{for all } k \geq 1.
\end{equation}
Analogous to \eqref{eq:projection-v} and \eqref{eq:Wk-regularity-v}, we respectively have
\begin{align}
	&
	\bar u(x) = \proj_{(-\infty, 0]}\left[ \bar \omega_1(x) - H_1(-g_1(x,\bar y(x))) \right]+H_1(-g_1(x,\bar y(x))) \quad \text{for a.e. } x \in \Omega \label{eq:projection-u} \\
	\intertext{and}
	& \bar u \in W^{1,k}(\Omega) \quad \text{for all } k \geq 1, \label{eq:W1k-regularity-u}
\end{align}
where $\omega_1 := \Delta_1^{-1}(-\bar\phi)$. 
From \eqref{eq:Wk-regularity-v} and \eqref{eq:W1k-regularity-u}, the regularity of solutions to \eqref{eq:state} (see; e.g. \cite{Grisvard1985}) yields 
\begin{equation}
	\label{eq:W2k-regularity-state}
	\bar y \in  W^{2,k}(\Omega) \quad \text{for all } k\geq 1,
\end{equation}
which together with the Sobolev embedding \cite{Adams,Brezis2010} gives the Lipschitz continuity on $\overline\Omega$ of $\bar y$. This, \cref{ass:g-funcs}, and \eqref{ass:Lip-L} as well as \eqref{ass:Lip-l} guarantee the Lipschitz continuity of functions $L_y'(\cdot, \bar y(\cdot))$, $g_{1y}'(\cdot,\bar y(\cdot))$ on $\overline\Omega$ and $\ell_y'(\cdot,\bar y(\cdot))$, $g_{2y}'(\cdot,\bar y(\cdot))$ on $\Gamma$. To derive the $W^{2,k}$-regularity of $\bar\phi$, we need to show that
\begin{equation}
	\label{eq:psi2-regularity}
	\psi_2 \in W^{\frac{1}{k'},k}(\Gamma) \quad \text{for all } k \geq 1.
\end{equation}
To this purpose, we deduce from \eqref{eq:stationary-v} and the definition of $H_2$ that
\begin{equation}
	\label{eq:psi2-express}
	\begin{aligned}[b]
		\psi_2 & = -\frac{1}{\zeta_2'(\bar v)} \left[ \gamma \bar\phi + \Delta_2(\bar v) \right] \\
		& = - H_2'(\zeta_2(\bar v)) \left[ \gamma \bar\phi + \Delta_2(\bar v) \right].
	\end{aligned}
\end{equation}
Thanks to \eqref{eq:Wk-regularity-v}, the fact that $\bar v \in L^\infty(\Gamma)$, \cref{lem:chain-rule}, and \cref{rem:chain-rule}, we deduce from \eqref{eq:lipzeta}, \eqref{eq:H-der-Lipschitz} and the definition of $\Delta_2$ that
\[
	\Delta_2(\bar v), H_2'(\zeta_2(\bar v)) \in W^{\frac{1}{k'},k}(\Gamma) \quad \text{for all } k \geq 1.
\]
Besides, \eqref{eq:W1k-regularity-phi} and the trace operator give $\gamma \bar \phi \in W^{\frac{1}{k'},k}(\Gamma)$ for all $k \geq 1$. Applying \cref{lem:production} to \eqref{eq:psi2-express} for $\tau = 1 - \frac{1}{k}, \tau_1 =\tau_2 = 1- \frac{1}{2k}$ and $k_1 = k_2 =2k$ finally yields \eqref{eq:psi2-regularity}. 

From \eqref{eq:psi2-regularity} and the fact that $\psi_1 \in L^\infty(\Omega)$ according to \eqref{eq:L-infty-psi1}, we can conclude that the right hand-side term of \eqref{eq:adjoint-state} belong to $L^\infty(\Omega) \times W^{\frac{1}{k'}, k}(\Gamma)$ for all $k\geq 1$.
The regularity of solutions to \eqref{eq:adjoint-state} again implies that
\begin{equation}
	\label{eq:W2k-regularity-adjoint}
	\bar \phi \in  W^{2,k}(\Omega) \quad \text{for all } k\geq 1.
\end{equation}
This ensures the Lipschitz continuity of $\bar\phi$ on $\overline\Omega$. Since $\bar\omega_i = \Delta_i^{-1}(-\bar\phi)$ and $\Delta_i^{-1}$, $i=1,2$, are Lipschitz continuous, we derive from \eqref{eq:projection-u}, \eqref{eq:projection-v} and the Lipschitz continuity of $\bar y$ that $\bar u$ and $\bar v$ are, respectively, Lipschitz continuous on $\overline{\Omega}$ and $\Gamma$. Furthermore, the Lipschitz continuity of  $\psi_2$ thus follows from \eqref{eq:psi2-express} and the Lipschitz continuity of $\bar \phi$, $\bar v$. Similarly, we derive the Lipschitz continuity of $\psi_1$.
$\qed$


\section*{Acknowledgements}
{This work was funded by Vietnam National Foundation for Science and Technology Development (NAFOSTED) under grant 101.01-2019.308. A part of this paper was completed at Vietnam Institute for Advanced Study in Mathematics (VIASM). The first author would like to thank VIASM for their financial support and
	hospitality.}


\printbibliography

@Article{ZoweKurcyusz1979,
	title =	 {Regularity and stability for the mathematical programming problem in {B}anach spaces},
	author =	 {Zowe, J. and Kurcyusz, S.},
	doi =		 {10.1007/BF01442543},
	journal =	 {Appl. Math. Optim.},
	number =	 {},
	volume =	 {5},
	year =	 {1979},
	pages =	 {49--62},
}

@book{Troltzsch2010,
	author = {Tr\"{o}ltzsch, Fredi},
	year = {2010},
	doi = {10.1090/gsm/112},
	publisher = {American Mathematical Society},
	series = {Graduate Studies in Mathematics},
	title = {Optimal Control of Partial Differential Equations: Theory, Methods and Applications},
	volume = {112},
}

@Article{Troltzsch2005,
	author = {Tr\"{o}ltzsch, Fredi},
	year = {2005},
	doi = {10.1137/S1052623403426519},
	journal = {SIAM J. Optim.},
	number = {2},
	pages = {616--634},
	title = {Regular Lagrange multipliers for control problems with mixed pointwise control-state constraints},
	volume = {15},
}

@Article{RoschTroltzsch2007,
	author = {R\"{o}sch, Arnd and Tr\"{o}ltzsch, Fredi},
	year = {2007},
	doi = {10.1137/060671565},
	journal = {SIAM J. Control Optim.},
	number = {3},
	pages = {1098--1115},
	title = {On regularity of 	solutions and Lagrange multipliers of optimal control problems for
	semilinear equations with mixed pointwise control-state constraints},
	volume = {46},
}

@Article{RoschTroltzsch2006,
	author = {R\"{o}sch, Arnd and Tr\"{o}ltzsch, Fredi},
	year = {2006},
	doi = {10.1137/050625114},
	journal = {SIAM J. Control Optim.},
	number = {2},
	pages = {548--564},
	title = {Existence of regular {L}agrange multipliers for a nonlinear elliptic optimal control problem with pointwise control-state constraints},
	volume = {45},
}

@Article{Robinson1976,
	title =	 {Stability theory for systems of inequalities, part II: Differentiable nonlinear systems},
	author =	 {Robinson, Stephen M.},
	doi =		 {10.1137/0713043},
	journal =	 {SIAM J. Numer. Anal.},
	number =	 {4},
	volume =	 {13},
	year =	 {1976},
	pages =	 {497--513},
}

@Article{MaurerZowe1979,
	title =	 {First- and second-order necessary and sufficient optimility conditions for infinite-dimensional programming problems},
	author =	 {Maurer, H. and Zowe, J.},
	doi =		 {10.1007/BF01582096},
	journal =	 {Math. Program.},
	number =	 {},
	volume =	 {16},
	year =	 {1979},
	pages =	 {98--110},
}

@Book{Kufner1977,
	author = {Kufner, Alois and John, Old\v{r}ich and Fu\v{c}\'{i}k, Svatopluk},
	year = {1977},
	note = {Monographs and Textbooks on Mechanics of Solids and Fluids; Mechanics: Analysis},
	publisher = {Noordhoff International Publishing, Leyden; Academia, Prague},
	title = {Function Spaces},
}

@Article{KienHuong2021,
	author =	 {Kien, Bui Trong and Huong, Nguyen Thi Thu and Qin, Xiaolong and Wen, Ching-Feng and Yao, Jen-Chih},
	title =	 {Regularity of solutions to a distributed and boundary optimal control problem governed by semilinear elliptic equations},
	doi =		 {10.1016/j.jmaa.2020.124694},
	journal =	 {J. Math. Anal. Appl.},
	number =	 {1},
	volume =	 {495},
	year =	 {2021},
	pages =	 {124694},
}

@Article{KienNhuSon2017,
	author = {Kien, Bui Trong and Nhu, Vu Huu and Son, Nguyen Hai},
	year = {2017},
	doi = {10.1007/s11228-016-0373-8},
	journal = {Set-Valued Var. Anal},
	number = {1},
	pages = {177--210},
	title = {Second-order optimality conditions for a semilinear elliptic optimal control problem with mixed pointwise constraints},
	volume = {25},
}

@Book{Ioffe,
	author =	 {Ioffe, A. D. and Tihomirov, V. M.},
	title =	 {Theory of Extremal Problems},
	publisher =	 {North-Holand Publishing Company},
	address =	 {Amsterdam-New Jork-Oxford},
	year =	 {1979},
	edition =	 {1st}
}

@Book{Grisvard1985,
	author = {Grisvard, Pierre},
	year = {1985},
	doi = {10.1137/1.9781611972030},
	publisher = {Pitman Advanced Pub. Program},
	title = {Elliptic Problems in Nonsmooth Domains},
}

@Article{Griesse2010,
	title =	 {Local quadratic convergence of SQP for elliptic optimal control problems with mixed control-state constraints},
	author =	 {Griesse, Roland and Metla, Nataliya and R\"{o}sch, Arnd},
	doi =		 {},
	journal =	 {Control Cybernet.},
	number =	 {},
	volume =	 {39},
	year =	 {2010},
	pages =	 {717--738},
}

@Article{GiangTuanSon2020,
	author =	 {Giang, Nguyen Bang and Tuan, Nguyen Quoc and Son, Nguyen Hai},
	title =	 {Regularity of Lagrange multipliers for mixed optimal control problem governed by semilinear elliptic equations},
	doi =		 {10.1007/s11117-020-00793-3},
	journal =	 {Positivity},
	number =	 {},
	volume =	 {},
	year =	 {2020},
	pages =	 {online},
}

@Book{Evans1992,
	author =	 {Evans,Lawrence Craig and Gariepy, Ronald F.},
	title =	 {Measure Theory and Fine Properties of Function},
	publisher =	 {CRC Press},
	address =	 {New York},
	year =	 {1992},
	edition =	 {4th},
	doi ={10.1201/b18333}
}

@Book{Cioranescu1990,
	author = {Cioranescu, Ioana},
	publisher = {Kluwer Academic Publishers},
	year = {1990},
	doi = {10.1007/978-94-009-2121-4},
	series = {Mathematics and Its Applications},
	title = {Geometry of Banach Spaces, Duality	Mappings and Nonlinear Problems},
	edition ={},
}

@InBook{CasasRaymondZidani1996,
	author =       {Casas, Eduardo and Raymond, Jean-Pierre  and  Zidani, Housnaa},
	title =        {Optimal control problem governed by semilinear elliptic equations with integral control constraints and pointwise state constraints},
	booktitle =      {Control and Estimation of Distributed Parameter Systems},
	publisher =    {Birkh\"{a}user, Basel},
	year =         {1998},
	DOI =         {10.1007/978-3-0348-8849-3_7},
	pages =	 {89--102},
	chapter ={},
}

@Article{CasasMateos2002,
	author = {Casas, Eduardo and Mateos, Mariano},
	year = {2002},
	doi = {10.1137/S0363012900382011},
	journal = {SIAM J. Control Optim.},
	number = {5},
	pages = {1431--1454},
	title = {Second order optimality conditions for semilinear elliptic control problems with finitely many state constraints},
	volume = {40}
}

@Article{Casas1993,
	title =	 {Boundary control of semilinear elliptic
	equations with pointwise state constraints},
	author =	 {Casas, Eduardo},
	doi =		 {10.1137/0331044},
	journal =	 {SIAM J. Control and Optimization},
	number =	 {4},
	volume =	 {31},
	year =	 {1993},
	pages =	 {993--1006},
}

@Article{Casas1986,
	title =	 {Control of an elliptic problem with pointwise state constraints},
	author =	 {Casas, Eduardo},
	doi =		 {10.1137/0324078},
	journal =	 {SIAM J. Control Optim.},
	number =	 {6},
	volume =	 {24},
	year =	 {1986},
	pages =	 {1309--1318},
}

@Book{Brezis2010,
	author = {Brezis, Haim},
	year = {2011},
	doi = {10.1007/978-0-387-70914-7},
	publisher = {Springer},
	series = {Universitext},
	title = {Functional Analysis, {S}obolev Spaces and Partial Differential Equations},
	address = {New York},
	edition ={1st}
}

@Book{BonnansShapiro2000,
	author = {Bonnans, J. Fr\'{e}d\'{e}ric and Shapiro, Alexander},
	year = {2000},
	doi = {10.1007/978-1-4612-1394-9},
	publisher = {Springer},
	title = {Perturbation Analysis of Optimization Problems},
	address = {New York},
	edition ={1st}
}

@Article{BayenBonnansSilva2013,
	title =	 {Characterization of local quadratic growth for strong minima in the optimal control of semilinear elliptic equation},
	author =	 {Bayen, T\'{e}rence and Bonnans, J. Fr\'{e}d\'{e}ric and J. Silva, Francisco},
	doi =		 {10.1090/S0002-9947-2013-05961-2},
	journal =	 {Trans. Amer. Math. Soc.},
	number =	 {},
	volume =	 {366},
	year =	 {2014},
	pages =	 {2063--2087},
}

@Book{Adams,
	author =	 {Adams, Robert A.},
	title =	 {Sobolev spaces},
	publisher =	 { Academic Press},
	address =	 {New York},
	year =	 {1975},
	edition =	 {},
	series={Pure and applied mathematics},
}

\end{document}